\documentclass{article}

\usepackage{makeidx}
\usepackage{latexsym}
\usepackage{amsfonts}
\usepackage{amssymb}
\usepackage{amsmath}
\usepackage{amstext}
\usepackage{amsthm}
\usepackage{mathrsfs}
\usepackage{color}

\newtheorem{theorem}{Theorem}
\newtheorem{example}[theorem]{Example}

\newtheorem{lemma}[theorem]{Lemma}
\newtheorem{remark}[theorem]{Remark}
\newtheorem{corollary}[theorem]{Corollary}
\newtheorem{definition}[theorem]{Definition}

\newtheorem{proposition}[theorem]{Proposition}

\newtheorem*{problem1}{Anonymity Problem}
\newtheorem*{problem2}{Neutrality Problem}
\newtheorem*{problem3}{Symmetry Problem}

\newtheorem*{openproblem}{Open Problem}

\begin{document}

\title{Symmetry groups for social preference functions }

\author{\textbf{Daniela Bubboloni}\footnote{Supported by GNSAGA of INdAM (Italy).
} \\
{\small {Dipartimento di Matematica e Informatica U.Dini} }\\
\vspace{-6mm}\\
{\small {Universit\`{a} degli Studi di Firenze} }\\
\vspace{-6mm}\\
{\small {viale Morgagni 67/a, 50134 Firenze, Italy}}\\
\vspace{-6mm}\\
{\small {e-mail: daniela.bubboloni@unifi.it}}\\
\vspace{-6mm}\\
{\small https://orcid.org/0000-0002-1639-9525} \and \textbf{Francesco Nardi}
\\
\\
{\small {e-mail: fardifran@gmail.com }}\\
\vspace{-6mm}\\
{}\\
\vspace{-6mm}\\}

\maketitle

\begin{abstract}
\noindent We introduce the anonymity group, the neutrality group and the symmetry group of a social preference function. Inspired by a problem posed by Kelly in 1991 and remained unsolved, we investigate the problem of recognizing which permutation groups may arise as anonymity, neutrality and symmetry group of a social preference function. 
 A complete description is found for the neutrality groups and a sufficient condition, which largely  encompasses the problem, is found for the anonymity groups.
 Using the concept of orbit extension of a group $U$, we formulate manageable necessary conditions for being $U$ an anonymity or a symmetry group. Our research deeply interacts with problems of representability by Boolean functions shedding light on them.
 \end{abstract}

\vspace{4mm}

\noindent \textbf{Keywords:} social choice theory; group theory; Boolean functions.

\vspace{2mm}

\noindent \textbf{MSC classification:} 91B14, 20B05, 06E30.

\section{Introduction}

The main object of social choice theory is the description of the aggregation of individual preferences into a collective choice. Individual preferences are expressed by relations on a set $N$ of $n\geq 2$ alternatives and it is usually required that those relations are orders, that is, reflexive, transitive and complete. Very often it is actually required that they are also antisymmetric, that is, linear orders. The ordered list of the individual preferences is called the preference profile.
The collective choice can be many different things: a set of relations on $N$, a set of linear orders on $N$, a set of $k$-sets  of $N$ for some $k\geq 1$.  The aggregation is given by some procedure $F$ given by a function from the preference profile set $\mathcal{P}$ into the set of desired collective choices. For all the possible instances  of collective choice it is considered a remarkable property for $F$ to be resolute, that is, to have as output a singleton. For instance, a social choice function ({\sc scf}) is defined as a procedure for which the individual preferences are linear orders and the collective choice is a single alternative; a social preference function ({\sc spf}) is defined as a procedure for which the individual preferences are  linear orders and  the collective choice is a single linear order.

Many properties can be introduced to qualify an aggregation procedure. Anonymity, neutrality, efficiency, Condorcet consistency and strategy-proofness are only few possible examples.  Among them two basic requirements to which it is very hard to renounce are certainly anonymity and neutrality.  Anonymity  expresses the fact that the names of the individuals are immaterial and do not impact the collective choice; neutrality expresses the fact that the names of the alternatives are immaterial and do not impact the collective choice. Unfortunately, usually  anonymity and neutrality cannot be guaranteed for resolute aggregation procedures. 
	
Let $h\geq 2$ be the number of individuals. A famous theorem by Moulin (\cite[Theorem 1, p.25]{mou})  establishes that there exists an anonymous, neutral and efficient {\sc scf} if and only if
\begin{equation}\label{moulin-condition}
\gcd(h,n!)=1,
\end{equation}
a very demanding arithmetical condition. Recall that a {\sc scf} is said efficient if, for every preference profile, it does not select an alternative which is unanimously beaten by another alternative.
If one renounces to efficiency, the existence of an anonymous and neutral  {\sc scf}  is equivalent to ask that the number $n$ of alternatives cannot be written as sum of non-trivial divisors of $h$ (\cite[Problem 1, p.25]{mou}), a slightly less demanding arithmetical condition. 

With no doubt  \eqref{moulin-condition} seems to be  a very crucial and pervasive condition for the description of anonymity and neutrality of aggregating procedures. For instance, as a further example relevant for our research, Bubboloni and Gori in \cite[Theorem 5]{BG14} prove that there exists an anonymous and neutral {\sc spf} if and only if \eqref{moulin-condition} holds. One possible explanation  for \eqref{moulin-condition} is surely the group theoretical interpretation of anonymity and neutrality proposed in   \cite{BG14} and in \cite{BG15},  where the fundamental notion of regular group is introduced and investigated.

It is also clear that even when the arithmetic link between $h$ and $n$ allows to design at least one aggregation procedure satisfying  anonymity and neutrality, it can happen that a specific aggregation procedure, which for some reason is considered appropriate for a certain decisional process,  fails  anonymity and neutrality. Recall, for instance, that when we break ties assigning the role of president to one of the individuals involved in the decisional process we necessarily destroy the anonymity. Similarly, if we break ties through a tie-breaking agenda, that is, an exogenously given ranking of the alternatives, then we necessarily destroy 
neutrality. In other words, typically, a resolute aggregation procedure satisfies only weak versions of anonymity and neutrality.

Those considerations naturally lead to two issues which we aim to address:
\begin{itemize}
\item[$(a)$] Given an aggregation procedure $F$ of a certain type, how can we  define a mathematical object capable to fully describe the anonymity and neutrality level of $F$?
\item[$(b)$] Given a specific level of anonymity and neutrality, when is it possible to design an aggregation procedure $F$ of a certain type having such a level of anonymity and neutrality?
\end{itemize}
The above questions are surely not a novelty in social choice literature. Our main reference and source of inspiration is given by the open problems set by
Kelly in \cite{Kelly91} in the context of aggregation procedures in which the individual preferences are orders (not necessarily linear) and the collective choice is a set of alternatives to be chosen within a certain subset of $N$. He considers, within that family, the subfamily in which the individual preferences are linear orders and reserves a particular consideration to  the resolute case. Moreover, in relation to $(a)$, Kelly defines separately a group of anonymity and a group of neutrality and  puts in evidence that question $(b)$ appears very hard to manage. That difficulty holds even taking into account some suitable subfamilies as, for instance, the {\sc scf}s. The fact that the problems set in $(b)$ are as interesting as hard is confirmed by the content of 
 \cite{Kelly92}, where Kelly gives some partial answers only for the  case in which the level of anonymity and neutrality is expressed through an abelian group. However, a general framework to attack the entire problem seems missing.

The goal of this paper is to give such a framework and answer to the issues $(a)$ and $(b)$ for {\sc spf}s. 
This is part of a wider ambitious project that would concern, in further papers, the {\sc scf}s and other families  of collective choices.
The algebraic instruments developed by Bubboloni and Gori (\cite{BG14},\cite{BG15},\cite{BG16}) reveal here all their power and give the main tool for the research, allowing the construction of the general framework.

While dealing with this problem in social choice theory, we immediately realized a main surprisingly connection between the anonymity problem for {\sc spf}s and the representability by Boolean functions, a  topic having a long history (see, for instance, \cite{CK},\cite{Kisie98},\cite{Gre10},\cite{KisieGre14},\cite{HM},\cite{HM21},\cite{KisieGre19}) and arising in the study of parallel complexity of formal languages. Surely our methods and results impact on classic representability questions, and we give some contribution throughout the paper. In particular, we formulate an original necessary condition for being $2$-representable (Corollary \ref{O-repres}).

The structure of the paper is the following.  In Section \ref{sec-one}, we give the basic notation and recall the fundamental theorems by Bubboloni and Gori (\cite{BG14},\cite{BG15},\cite{BG16}). In Section \ref{groups-asso}, in the view of $(a)$, we introduce the anonymity group, the neutrality group and the symmetry group of a social preference function and explore their elementary properties. Then, in the view of $(b)$, we define the concept of anonymity group, neutrality group and  symmetry group and pose the corresponding anonymity problem,  neutrality problem and symmetry problem. In Sections \ref{firstSteps} and \ref{anonimityESymmetry} we explore the link between the symmetry groups and the concept of regularity and show that every anonymity group is also a  symmetry group. In section \ref{sol-neut} we completely solve the neutrality problem, showing that every group of permutations of the alternatives can appear as neutrality group of a certain social preference function (Theorem  \ref{neutralita}). In Section \ref{sec-anonim} we explore the  harder anonymity problem giving a complete answer in the case $h\leq n!$ (Corollary \ref{anon-hminore}) and thus identifying the anonymity groups which arise allowing the number of alternatives to vary. In Section \ref{crt-ex} we discuss a crucial example of a neutrality group which is not a symmetry group.
In Section \ref{subsec:funzioni booleane} we make clear the connection between the anonymity problem for {\sc spf}s and the representability by  Boolean functions.  Finally in Section \ref{capitolo O(U)} we provide a necessary condition for symmetry and anonymity (Theorem \ref{necessh}), by introducing the concept of orbit extension.

\section{Preliminary definitions and results}\label{sec-one}

 Given $k\in \mathbb{N}$, and $r\in \mathbb{N}$ a prime we denote by $k_r$ the so-called $r$-part of $k$, that is the maximum $r$-power dividing $k$. We also define $[k]:=\{y\in \mathbb{N}\mid y\leq k\}$ and $[k]_0:=\{y\in \mathbb{N}\cup\{0\}\mid y\leq k\}.$ 
 
In this paper we are exclusively interested in the finite groups that arise as subgroups of direct products of two symmetric groups, the first one dealing with permutations of individuals names  and the second one dealing with permutations of alternatives names. 
Given a  nonempty finite set $X$, we denote by $S_X$  the symmetric group on $X$. If $\sigma_1,\sigma_2\in S_X$, the permutation $\sigma_1\sigma_2$ is such that $\sigma_1\sigma_2(x)=\sigma_1(\sigma_2(x)),$ for all $x\in X.$ Moreover, the  conjugate of  $\sigma_1$ by $\sigma_2$ is given by $\sigma_1^{\sigma_2} := \sigma_2\sigma_1\sigma_2 ^{-1}. $
If $X=[k]$ for some $k\in \mathbb{N}$, then  we use the more compact notation $S_k.$
Let $\sigma\in S_k$. The orbit of $j\in [k]$ under the action of $\langle \sigma\rangle$ is denoted by $j^{\langle \sigma\rangle}$. The orbits  of $\sigma$, that is the set $\{j^{\langle \sigma\rangle}:j\in [k]\}$,  give a partition of $[k]$ and thus their sizes $x_1,\dots,x_r$ sum up to $k.$  We call  type of $\sigma$ the unordered list $T_{\sigma}:=[x_1,\dots,x_r]$. We also set $\gcd(T_{\sigma}):=\gcd\{x_1,\dots, x_r\}$  and $\mathrm{lcm}(T_{\sigma}):=\mathrm{lcm}\{x_1,\dots, x_r\}.$
Note that the number of terms equal to $1$ in $T_{\sigma}$ counts the fixed points of $\sigma.$ 
Recall that two permutations are conjugate if and only if they have the same type and that $o(\sigma)=\mathrm{lcm} (T_{\sigma}).$

If $\mathcal{A}$ is a set of subgroups of a group $G$, we say that $\mathcal{A}$ is  subgroup-closed if $W\in \mathcal{A}$  and $U\leq W$ imply $U\in \mathcal{A}$; conjugacy-closed if $W\in \mathcal{A}$  and $y\in G$ imply $W^y\in \mathcal{A}.$

\subsection{ The action of the group $G$ on preference profiles} 

Let $H:=[h]$ with $h\geq 2$ be the set of the names of individuals and $N:=[n]$, with $n\geq 2$ be the set of the names of alternatives. We call the pair $(h,n)$ a voting pair.
Throughout  the paper we denote by $G$ the group $G:=S_h\times S_n$. 

The functions  $\pi_1: G \rightarrow S_h$ and  $\pi_2: G \rightarrow S_n$ defined, respectively, by $\pi_1(\varphi,\psi)= \varphi$ 
and  $\pi_2(\varphi,\psi)= \psi$ for all $(\varphi,\psi)\in G$,
are group homomorphisms called  projection on the first and second component respectively. 
 Let $g \in G$. Then we have $g=(\pi_1(g), \pi_2(g))$. We set $\varphi_g:=\pi_1(g)$ and $\psi_g:=\pi_2(g)$ in order to get the more coincise writing $g=(\varphi_g,\psi_g)$. 
  
Let  $U \leq G$. Then we have $\pi_1(U) \leq S_h$, $\pi_2(U) \leq S_n$ and $U \leq \pi_1(U) \times \pi_2(U)$ with inclusion generally strict.
Among the subgroups of $G$ we are  interested in $\mathcal {D}_G:=\{V \times W \leq G \mid  V \leq S_h \, \, \text{and} \, \, W \leq S_n  \}.$
Note that $U\in \mathcal {D}_G$ if and only if $U=\pi_1(U)\times \pi_2(U).$

Let $\mathcal{L}(N)$ be the set of linear orders on $N$.
As explained in detail in \cite[Section 2.2]{BG15}, there is a natural identification of $\mathcal{L}(N)$ with $S_n$ and we deeply rely on that. For instance, the linear order $3>2>4>1$ is identified with $\sigma\in S_4$ such that $\sigma(1)=3, \sigma(2)=2, \sigma(3)=4, \sigma(4)=1$ , that is, with $\sigma=(134).$

We then assume that each individual $i\in H$ expresses her  preferences through some $p_i\in S_n$. The corresponding  preference profile is given by $p=(p_i)^h_{i=1}\in \mathcal{L}(N)^h=(S_n)^h$. We denote the set of preference profiles
 by $\mathcal{P}_h(n):=\mathcal{L}(N)^h$ or more simply by $\mathcal{P}$, when we do not need to emphasize the dependence on $(h,n)$.
  
A preference profile $p \in \mathcal{P}$ is called costant if  there exists $\sigma\in S_n$ such that, for every $i \in H$, $p_i=\sigma$. We denote such $p$ by $p_{\sigma}$. The subset of $\mathcal{P}$ formed by the constant preference profiles is denoted by $\mathcal{K}.$

For every $(\varphi,\psi)\in G$ and $p\in \mathcal{P}$, we denote by  $p^{(\varphi,\psi)} \in \mathcal{P}$ the preference profile such that, for every $i\in H$,
\begin{equation*}\label{action}
\left(p^{(\varphi,\psi)}\right )_i:=\psi p_{\varphi^{-1}(i)}.
\end{equation*}

Remarkably, that definition, introduced in \cite{BG15}, determines an action of the group $G$ on $\mathcal{P}$. 
\begin{proposition}{\rm\cite[Proposition 2]{BG15}}
\label{prop:azio}
Let  $U \leq G$ and  $p \in \mathcal{P}$. 
\begin{enumerate}
\item[$(i)$] For every  $(\varphi_1, \psi_1), (\varphi_2,\psi_2) \in U$, we have
\begin{equation}
\label{azione} 
p^{(\varphi_1\varphi_2,\psi_1\psi_2)}=\big(p^{(\varphi_2,\psi_2)}\big)^{(\varphi_1,\psi_1)}.
\end{equation}
\item[$(ii)$] The function  $f : U \rightarrow S_{\mathcal{P}}$ defined, for every $(\varphi,\psi) \in U$, by 
\[
f(\varphi,\psi): \mathcal{P} \rightarrow \mathcal{P},\quad \quad p \mapsto p^{(\varphi,\psi)}
\]
is an action of the group $U$on the set $\mathcal{P}.$
\end{enumerate}
\end{proposition}

As a consequence we can use the typical objects related to an action. To start with, for $U\leq G$, we consider the $U$-orbit of $p\in \mathcal{P}$ given by 
$p^U:=\{p^g\in \mathcal{P}: g\in U\}$. 
The set  $\mathcal{P}^U=\{p^U:p\in\mathcal{P}\}$ of the $U$-orbits is a partition of $\mathcal{P}$. We denote the size of $\mathcal{P}^U$ by $R(U)$. When only a single subgroup $U$ comes into play, we write simply  $R$ instead of $R(U)$. 
Any ordered list  $(p^j)_{j=1}^{R}\in\mathcal{P}^{R}$ such that $\mathcal{P}^U=\{p^{j\,U} : j\in [R]\}$,
is called a system of representatives of the $U$-orbits. We denote the set of the systems of representatives of the $U$-orbits by $\mathfrak{S}(U)$.
For $p\in\mathcal{P}$, the stabilizer of $p$ in $U$ is the subgroup of $U$ defined by
$\mathrm{Stab}_U(p):=\{g\in U : p^g=p \}.$
As well-known, the size of $p^U$ is equal to the index of $\mathrm{Stab}_U(p)$ in $U.$

\subsection{$U$-Symmetric {\sc spf}s}
A  {\it social preference function} {\sc spf} on $n$ alternatives and $h$ individuals
 is a function from $\mathcal{P}=\mathcal{L}(N)^h$ to $\mathcal{L}(N)=S_n$. We denote
the set of  {\sc spf}s for the voting pair $(h,n)$  by $\mathcal{F}_h(n)$. If we do not need to emphasize the dependence on $(h,n)$, we  use the notation $\mathcal{F}$ instead of $\mathcal{F}_h(n)$.

We exhibit an easy example for the reader unfamiliar with the topic.
\begin{example}\label{majority}
{\rm  Assume that there are only $n=2$ alternatives so that $N=\{1,2\}$, and $h$ individuals. Then the possible preferences for the individuals are only two: $1>2$ and $2>1$. In terms of the symmetric group $S_2$, those preferences are, of course, $id$ and $(12).$ For a preference profile $p\in (S_2)^h$, let $c_{p}(1,2):=|\{i\in [h]:p_i=id\}|$ be the number of individuals which express the preference $1>2$.
We define the majority {\sc spf} with ties broken by $1>2$, as the function
$M:(S_2)^h\rightarrow S_2$ given, for every $p\in (S_2)^h$, by 
\[M(p)=
\begin{cases}
(12) \quad  \text{if } \quad c_{p}(1,2)<h/2\\
id \quad \quad \text{if } \quad c_{p}(1,2)\geq h/2.\\
\end{cases}
\]
Note that when $h$ is even and exactly $h/2$ individuals express the preference $1>2$ and $h/2$ individuals express the preference $2>1$, $M$ selects $1>2.$}
\end{example}
\begin{definition}\label{u-sym}
\rm{Let $F\in \mathcal{F}$. Given a subgroup $U$ of $G$, 
we say that $F\in \mathcal{F}$ is $U$-{\it symmetric} if, for every $p\in \mathcal{P}$ and $(\varphi,\psi)\in U$,
\[
F(p^{(\varphi,\psi)})=\psi F(p),
\]
where the product in the right side is the product of the permutations $\psi$ and $F(p)$ inside the group $S_n.$
The set of $U$-symmetric  {\sc spf}s is denoted by $\mathcal{F}_h(n)^U$, or more simply by $\mathcal{F}^U$ when we do not need to emphasize the dependence on $(h,n)$.} \end{definition}
For instance, consider the {\sc spf} $M$ in Example \ref{majority}. If $h$ is even, then $M$  is $S_h\times\{id\}$-symmetric. However, $M$ is not $S_h\times S_2$-symmetric. That relies on the fact that 
 ties are broken privileging the alternative  $1$ with respect to the alternative $2.$ If $h$ is odd, we instead see that $M$ is $S_h\times S_2$-symmetric. 
\smallskip

As a particular case of symmetry we recover anonymity and neutrality. In fact,
 $F\in \mathcal{F}$ is called \textit{anonymous} if  $F \in \mathcal{F}^{S_h \times \lbrace id \rbrace}$; \textit{neutral} if $F \in \mathcal{F}^{ \lbrace id \rbrace \times S_n}.$
 We recall a useful result.
 \begin{proposition} {\rm\cite[Proposition 1]{BG15}} \label{generato_intersezione}
  If $U_1,U_2 \leq G$, then $\mathcal{F}^{U_1} \cap \mathcal{F}^{U_2}= \mathcal{F}^{\langle U_1,U_2\rangle}$. In particular,  $F$ is anonymous and neutral if and only if $F\in \mathcal{F}^G$.
\end{proposition}
Note that if $U'\le U$, then $\mathcal{F}^{U}\subseteq  \mathcal{F}^{U'}.$
Note also that $\mathcal{F}^1=\mathcal{F}$ is never empty. However, for a generic $U\leq G$, the existence of a $U$-{\it symmetric} {\sc spf}  is not guaranteed. To that issue is  devoted the next section.
 
\section{Regular subgroups of $G$}\label{new-reg}

\subsection{Basic facts}
In order to address the possibility of having $\mathcal{F}_h(n)^U\neq \varnothing$  for some voting pair $(h,n)$ and understand the concreteness of construction of $U$-symmetric  {\sc spf}s, we recall a fundamental definition from \cite{BG15}.
A subgroup $U$  of $G=S_h\times S_n$ is said {\it regular} if, for every $p\in\mathcal{P}$,
\begin{equation}\label{property}
\begin{array}{c}
\vspace{-2mm}\\
\mathrm{Stab}_U(p)\subseteq S_h\times \{id\}.
\end{array}
\end{equation}
The set of regular subgroups of $G$ is denoted by $\mathcal{R}_h(n)$ or more simply by $\mathcal{R}$, when we do not need to emphasize the dependence on $(h,n)$. We stress  that the word regular is used with a different meaning in permutation group theory,  indicating a transitive permutation group with trivial stabilizer. The same word is used in \cite{BG15} for some analogy,  because the sugroups satisfying \eqref{property} have stabilizers with trivial projection on the factor $S_n$.
 \begin{proposition}\label{Rsubclosed} $\mathcal{R}$  is nonempty and subgroup-closed. In particular, $G$ is regular if and only if each subgroup of $G$ is regular.  
\end{proposition}
 \begin{proof} The subgroups of $G$ of the form $V \times \lbrace id \rbrace$, for some $V \leq S_h$, are regular so that $\mathcal{R}\neq \varnothing.$
Let $U\in \mathcal{R}$  and $W\leq U$. Then we have  $\mathrm{Stab}_W(p)\leq\mathrm{Stab}_U(p)\subseteq S_h\times \{id\},$ so that $W\in \mathcal{R}$.  
 \end{proof}

 An immediate consequence of the definition, is the following useful lemma.
\begin{lemma}\label{banale-reg}
Let $g_1=(\varphi_1,\psi_1),g_2=(\varphi_2,\psi_2) \in G$. If $\langle g_1,g_2 \rangle\in \mathcal{R}$ and $p \in \mathcal{P}$, then 
$p^{g_1}=p^{g_2}$ implies $ \psi_1=\psi_2.$
\end{lemma}

We now recall two useful  main results from the literature.
\begin{theorem} {\rm\cite[Theorem $7$]{BG15}}\label{teofon}
$\mathcal{F}^U \neq \varnothing$ if and only if $U \in \mathcal{R}$. 
\end{theorem}

\begin{theorem}{\rm\cite[Propositions $3$ and $4$]{BG14}}
\label{costruzione}
Let  $U \in \mathcal{R}$, $(p^i)_{j=1}^{R} \in \mathfrak{S}(U)$ be a system of representatives for the $U$-orbits on  $\mathcal{P}$ and $\bold{q}=(q_i)_{j=1}^{R} \in \mathcal{L}(N)^{R}$. Then the following facts hold:
\begin{itemize}
\item[$(i)$] there exists a unique $F \in \mathcal{F}^{U}$ such that, for every $j \in [R]$, $F(p^j)=q_j$ holds.
\item[$(ii)$] Once denoted by $F_{\bold{q}}$  the function in $(i)$, the map 
 \[
 g: \mathcal{L}(N)^{R} \rightarrow \mathcal{F}^{U}, \quad  g(\bold{q})= F_{\bold{q}}
 \]
 is a bijection. In particular, $|\mathcal{F}^{U}|=n!^R$.

\end{itemize}
\end{theorem}
\subsection{The class $\mathcal{R}$}
The group theoretical characterization of regularity comes as a particular
case of  \cite[ Theorem $12$]{BG15}.

\begin{theorem} \label{Caratterizzazione}
Let $U \leq G$. The following facts are equivalent:
\begin{enumerate}
\item[$(a)$]$U \in \mathcal{R}$;
\item[$(b)$] for every $(\varphi,\psi) \in U$ with $\psi \neq id$ and $r$ a prime number such that  $|\psi|_{r}=r^a$ for some $a \in \mathbb{N}$, we have $r^a \nmid \gcd(T_\varphi)$.
\end{enumerate}
\end{theorem}
The above theorem makes the concept of regularity easy to manage and allows  to distinguish regular subgroups  from non-regular subgroups. For instance, $A_4 \times S_2 \notin \mathcal{R}$ because  $((12)(34),(12))\in A_4 \times S_2$ and, since $T_{(12)(34)}=[2,2]$, we have $2=|(12)| \mid \gcd[2,2]=2$. We now exhibit an interesting wide family of regular subgroups.

\begin{corollary}\label{reg-sub-Sn} Let $W\leq S_n$. If $V \leq S_{[h]\setminus\{i\}}$, for some $i\in [h]$, then $U=V \times W\in \mathcal{R}$. In particular
 $\lbrace id \rbrace \times W\in \mathcal{R}$.
\end{corollary}
\begin{proof} Let $(\varphi,\psi) \in U$ with $\psi \neq id$ and $r$ be a prime number such that  $|\psi|_{r}=r^a$ for some $a \in \mathbb{N}$. Since $\varphi(i)=i$ we have that $T_{\varphi}$ contains a term equal to $1$ and thus $\gcd(T_{\varphi})=1$ so that we surely have $r^a \nmid \gcd(T_{\varphi})$.
Thus, by Theorem \ref{Caratterizzazione}, we deduce that $U$ is  regular. The second part follows immediately because $\lbrace id \rbrace  \leq S_{[h]\setminus\{i\}}$.
\end{proof}

We also recall Lemma 17 in \cite{BG15}. It gives a necessary and sufficient condition for the regularity of the whole group $G=S_h\times S_n$.
\begin{proposition}\label{lemma17} The group $G$ is regular if and only if $\gcd(h,n!)=1.$
\end{proposition}

\noindent  We have already noticed that  $\mathcal{R}$ is subgroup-closed. Using Theorem \ref{Caratterizzazione} and recalling that  conjugation preserves the type and the order of a permutation,  the following proposition is easily proved. We omit the proof for the sake of brevity. 
 \begin{proposition}\label{R-conj-closed}
 The class $\mathcal{R}$ is conjugacy-closed.
 \end{proposition}

A main obstacle in building $ \mathcal{R}$ by the classic group theoretical operations is given by the following observation.
\begin{remark}\label{no-gen} The subgroup generated by two regular subgroups is not, in general, regular. 
\end{remark}
\begin{proof}
Consider the voting pair $(h,n)=(3,3)$ and the subgroups of  $G=S_3 \times S_3$ given by 
$V=\langle(123)\rangle \times \lbrace id \rbrace$ and $W=\lbrace id \rbrace \times \langle(123)\rangle$.
$V \in \mathcal{R}$ because it is  included in $S_3 \times \{ id \}$ while  $W \in \mathcal{R}$ by Corollary \ref{reg-sub-Sn}. However $\langle V,W \rangle= \langle (123) \rangle \times \langle (123) \rangle$
is not regular by Theorem \ref{Caratterizzazione}.
\end{proof}

\section{The symmetry groups associated with a {\sc spf} }\label{groups-asso}

We introduce now the main definitions of the paper.
\begin{definition} \label{def_gruppidisimmetria}
\noindent {\rm Let $F \in \mathcal{F}$. We define 
\begin{itemize}
\item[1)] the {\it anonymity group} of $F$ by
\[
G_1(F):= \big \lbrace (\varphi,id) \in G \mid F(p^{(\varphi,id)})=F(p), \quad \forall \, \, p \in \mathcal{P} \big \rbrace;
\]
\item[2)] the {\it neutrality group} of $F$ by
\[
G_2(F):= \big \lbrace (id,\psi) \in G \mid F(p^{(id,\psi)})= \psi F(p), \quad \forall \, \, p \in \mathcal{P} \big \rbrace;
\]
\item[3)] the {\it symmetry group} of $F$ by
\[
G(F):= \big \lbrace (\varphi,\psi) \in G \mid F(p^{(\varphi,\psi)})= \psi F(p), \quad \forall \, \, p \in \mathcal{P} \big \rbrace.
\]
\end{itemize}
We globally refer to the groups $G_1(F),\ G_2(F), G(F)$ as to the {\it symmetry groups} of  the {\sc spf} $F$.}
\end{definition}

We give a couple of examples.  Let $D_i$ be the dictatorship with dictator $i\in H$, that is the {\sc spf}  defined by $D_i(p)=p_i$ for all $p\in \mathcal{P}$.
Then we have
\begin{equation}\label{Di}
G_1(D_i)=\big \lbrace (\varphi,id) \in G \mid \varphi(i)=i \big \rbrace \times \{id\}\cong S_{H\setminus\{i\}}\times \{id\} , \qquad G_2(D_i)=\{id\} \times S_n,
\end{equation}
and $G(D_i)=S_{H\setminus\{i\}}\times S_n.$
Note that the group $S_{H\setminus\{i\}}\times \{id\}$ is maximal in $S_h\times \{id\}$ and thus the anonymity level of a dictatorship is very high. By similar considerations the symmetry level of a dictatorship is very high too. 

For the majority {\sc spf} $M$ considered in  Example \ref{majority}, we  have a different scenario depending on the parity of $h$.
If $h$ is even, then
$$G_1(M)=S_h\times\{id\}=G(M), \qquad G_2(M)=\{id\}\times \{id\}.$$
If $h$ is odd, then 
$$G_1(M)=S_h\times\{id\}, \qquad G_2(M)=\{id\}\times S_2, \qquad G(M)=S_h\times S_2.$$
By the above results we see that $M$ is anonymous and neutral if and only if $h$ is odd.

Note that, for every {\sc spf} $F$,  we obviously have
 \begin{equation}\label{ovvia}
G_1(F)= G(F) \cap \big(S_h \times \lbrace id \rbrace\big) , \quad G_2(F)= G(F) \cap \big(\lbrace id \rbrace \times S_n\big).
 \end{equation}

As said in the introduction, for Definition \ref{def_gruppidisimmetria} we are inspired to \cite{Kelly91}. However,  Kelly proposes only a separate study of anonymity groups and neutrality groups and does not define the symmetry group. Since nowadays, by \cite{BG15},  the presence of a group action is fully acknowledged, it is instead very natural and fruitful to treat these three concepts in a unitary form and focus on the relations among them. 
\begin{proposition}\label{legami}
Let  $F \in \mathcal{F}$. Then the following facts hold:
\begin{itemize}
\item[$(i)$] $G(F)$ is a subgroup of  $G$.
\item[$(ii)$]  $G_1(F)$ and $G_2(F)$ are subgroups of $G(F)$ with $G_1(F)\leq S_h \times \{id\}$ and $G_2(F)\leq  \{id\} \times S_n.$ 
 \item[$(iii)$]  $\langle G_1(F), G_2(F) \rangle = G_1(F) \times G_2(F) \leq G(F)$ and the inclusion is generally proper. 
 \item[$(iv)$] $G(F)=G_1(F) \times  G_2(F)$ if and only if $G(F) \in \mathcal{D}_G$.
\item[$(v)$] If $G(F) \leq S_h \times \{id\}$ then $G(F)=G_1(F)$. If $G(F) \leq \{id\} \times S_n$, then $G(F)=G_2(F)$.
 \end{itemize}
\end{proposition}
\begin{proof}
 $(i)$ Surely $(id,id)\in G(F)$ so that  $G(F) \neq \varnothing$. If $(\varphi_1,\psi_1),(\varphi_2,\psi_2)\in G(F)$ and  $p \in \mathcal{P}$, by \eqref{azione}  we have,
\[
F(p^{(\varphi_1 \varphi_2,\psi_1\psi_2)})=F((p^{(\varphi_2,\psi_2)})^{(\varphi_1,\psi_1)})=\psi_1F(p^{(\varphi_2,\psi_2)})=\psi_1\psi_2F(p).
\] 
 Thus $(\varphi_1,\psi_1)(\varphi_2,\psi_2)=(\varphi_1\varphi_2,\psi_1 \psi_2) \in G(F)$. Since we are dealing with finite groups this guarantees that  $G(F)$ is a subgroup of $G$. 
 \vspace{2mm}
 
 \noindent$(ii)$  Since by $(i)$ we have $G(F)\leq G$ and the intersection of subgroups is a subgroup,  the fact that $G_1(F)\leq G(F)$ and $G_2(F) \leq G(F)$ follow immediately by \eqref{ovvia}.
 \vspace{2mm}
 
 \noindent$(iii)$ The groups $S_h \times \lbrace id \rbrace$ and $\lbrace id \rbrace \times S_n$ permute element by element and have trivial intersection. Thus, by \eqref{ovvia}, the same holds for $G_1(F)$ and $G_2(F)$. Hence $\langle G_1(F), G_2(F) \rangle = G_1(F) G_2(F)=G_1(F) \times G_2(F).$

 We now give an example  of a proper containment between $G(F)$ and $G_1(F) \times  G_2(F)$. Consider the voting pair $(h,n)=(3,2)$ and $U:=\langle(\sigma,\sigma)\rangle$, where $\sigma:=(12).$ Since, by Proposition \ref{lemma17},  $G$ is regular, every subgroup of $G$ is regular and thus $U$ is regular. Thus, by Theorem \ref{teofon}, $\mathcal{F}^U\neq \varnothing.$
 We show that $U $ is the symmetry group of some $F \in \mathcal{F}^U$.
 Assume, by contradiction, that $G(F)>U$ for all $F \in \mathcal{F}^U$. Since the only subgroups containing properly $U$ are $V= \langle (\sigma,id),(id,\sigma) \rangle$ and $ G$ itself, we have $\mathcal{F}^U=\mathcal{F}^G \cup \mathcal{F}^V$. But since $\mathcal{F}^G\subseteq  \mathcal{F}^V$ we then have $\mathcal{F}^U=\mathcal{F}^V$ and thus $|\mathcal{F}^U|=|\mathcal{F}^V|$. By Theorem \ref{costruzione}, this means $2^{R_U}=2^{R_V}$ and hence $R_U=R_V$.
We show that  instead we have $R_V<R_U$. Indeed consider the preference profile $p:=(\sigma, id, \sigma)\in \mathcal{P}$ and note that $p^U=\{p, (\sigma, id, id)\}$ while $p^V=\{p, (\sigma, id, id),(id, \sigma, id), (id, \sigma,\sigma)\}$. Hence at least the $V$-orbit $p^V$ is union of more than one $U$-orbit, which implies $R_V<R_U$.
 Thus there exists $F\in \mathcal{F}$ such that $U=G(F).$
By \eqref{ovvia}, we then have $G_1(F)=\{id\}\times \{id\}=G_2(F)$ and thus  also $G_1(F) \times  G_2(F)=\{id\}\times \{id\}<U.$
 \vspace{2mm}
 
 \noindent$(iv)$ Assume that  $G(F)=G_1(F) \times G_2(F)$. Then, clearly, we also have that $G(F)=\pi_1(G_1(F)) \times \pi_2(G_2(F))$ and hence $G(F)\in \mathcal{D}_G.$
 
Conversely assume that $G(F)=V \times W\in \mathcal{D}_G$ for some $V \leq S_h$ and $W \leq S_n$. We show that  $G_1(F)=V \times \{ id \}$. Let  $(\varphi,id) \in G_1(F)$. Since $
G_1(F)\leq G(F)$, we have that $(\varphi,id) \in V \times W$ and thus $\varphi \in V$, so that  $(\varphi,id) \in V \times \{ id \}$. On the other hand, if $(\varphi,id) \in V \times \{ id \}$ we also have $(\varphi,id) \in G(F)$  which, by the definition of symmetry group, implies $(\varphi,id) \in G_1(F)$. Analogously one show that  $G_2(F)= \{id \} \times S_n$. It follows that $G(F)=V \times W = (V \times \{id \}) \times (\{ id \} \times W)=G_1(F) \times G_2(F).$ 
\vspace{2mm}

\noindent$(v)$ Those facts follow immediately by \eqref{ovvia}.
\end{proof}

An important consequence of Proposition \ref{legami} $(iii)$ is that the symmetry group $G(F)$ is not generally equal to the product $G_1(F) \times G_2(F)$ of the anonymity and neutrality groups. We emphasize that, the case $G(F) \notin \mathcal{D}_G$ is of interest in the applications as shown by the next example.

\begin{example} \label{noProduct}
	
\rm{	Assume that the alternatives are a subset of the individuals, that is, $N \subseteq H$.  This situation naturally arises when a committee needs to elect some of its members to form a subcommittee or to elect a president. Obviously we then have $n\leq h.$
In this case, from a social choice theory point of view, the only admissible changes for individuals and  alternative names are those belonging  to}:
$$ U_{N\subseteq H}:=\lbrace (\varphi,\psi) \in S_h \times S_n \mid \varphi(i)=\psi(i),  \quad \forall \, i \in N \rbrace$$

\noindent Of course $U_{N\subseteq H}$ is a subgroup of $G$ and it is immediately checked  that   $U_{N\subseteq H} \notin \mathcal{D}_{G}$. It would be nice to be able to design a {\sc spf} $F$ with $G(F) = U_{N\subseteq H}$. However we will see in Section \ref{firstSteps} that this is never possible.
\end{example}

We state, for completeness,  the fundamental result about the regularity of $ U_{H\subseteq N}$.
\begin{proposition}{\rm\cite[Proposition 22]{BG21}}\label{vecchia}
Let $N\subseteq H$. Then  $U_{N\subseteq H} \in \mathcal{R}$ if and only if $\gcd(h,n)=1$.
\end{proposition}
In particular, since $\gcd(h,h)=h\neq 1$, the diagonal group $U_{N= H}:=\{(\varphi,\varphi) \in S_h \times S_h\}$ is never regular. 
\begin{definition} \label{def_gruppodi}
\noindent {\rm Let $(h,n)$ be a voting pair. We say that 
\begin{itemize} 
\item[1)] $U\leq S_h \times \{id\}$ is an \textit{anonymity group} for {\sc spf}s, with respect to the voting pair $(h,n)$,  if there exists $F \in \mathcal{F}_h(n)$ such that $G_1(F)= U.$ The set of such anonymity groups is denoted by $${\rm SPFAG}_h(n);$$
\item[2)] $U\leq  \{id\}\times S_n  $ is a \textit{neutrality group} for {\sc spf}s, with respect to the voting pair $(h,n)$,  if there exists $F \in \mathcal{F}_h(n)$ such that $G_2(F)= U$ 
The set of such neutrality groups is denoted by $${\rm SPFNG}_h(n);$$

\item[3)] $U\leq G$ is a \textit{symmetry group} for {\sc spf}s, with respect to the voting pair $(h,n)$,  if there exists $F \in \mathcal{F}_h(n)$ such that $G(F)=U$.
The set of such symmetry groups is denoted by $${\rm SPFSG}_h(n).$$
\end{itemize}
We define next the following sets of groups obtained  by the above sets just letting $h$ and $n$ vary:
\begin{itemize}
\item [1)] 
${\rm SPFAG}_h:=\bigcup_{n}{\rm SPFAG}_h(n),\quad {\rm SPFAG}(n):=\bigcup_{h}{\rm SPFAG}_h(n),\quad {\rm SPFAG}:=\bigcup_{h,n}{\rm SPFAG}_h(n);$
\item[2)] 
${\rm SPFNG}_h:=\bigcup_{n}{\rm SPFNG}_h(n),\quad {\rm SPFNG}(n):=\bigcup_{h}{\rm SPFNG}_h(n),\quad {\rm SPFNG}:=\bigcup_{h,n}{\rm SPFNG}_h(n);$
\item[3)] 
${\rm SPFSG}_h:=\bigcup_{n}{\rm SPFSG}_h(n),\quad {\rm SPFSG}(n):=\bigcup_{h}{\rm SPFSG}_h(n),\quad {\rm SPFSG}:=\bigcup_{h,n}{\rm SPFSG}_h(n).$
\end{itemize}}
\end{definition}

A voting pair $(h,n)$ is called {\it  fully anonymous} if ${\rm SPFAG}_h(n)=\{U\leq S_h\times\{id\}\}$; {\it fully neutral } if ${\rm SPFNG}_h(n)=\{U\leq \{id\}\times S_n\}$; {\it  fully symmetric} if ${\rm SPFSG}_h(n)=\{U\leq S_h\times S_n\}$.

Note that the fact that a voting pair is fully anonymous (neutral, symmetric) reflects the fact that it is possible to construct {\sc spf}s having any  level of anonymity (neutrality, symmetry), a very remarkable property.

If the voting pair $(h,n)$  is fully anonymous (neutral, symmetric), then the set ${\rm SPFAG}_h(n)$ (${\rm SPFNG}_h(n),\ \ {\rm SPFSG}_h(n)$) is also subgroup-closed. However, for a generic pair $(h,n)$, the sets ${\rm SPFAG}_h(n)$, ${\rm SPFNG}_h(n)$ and $ {\rm SPFSG}_h(n)$  are not, in general,  subgroup-closed. 
That fact surely constitutes a main difficulty  in treating with them.

The nature of the anonymity groups, neutrality groups and symmetry groups is generally quite complex and it can be hard to decide if a certain group  belongs to some of those classes.
We pose then the three problems which our paper wants to address.

\begin{problem1}\label{an-pb}
Determine the groups in ${\rm SPFAG}_h(n)$,  ${\rm SPFAG}_h$,  ${\rm SPFAG}(n)$ and ${\rm SPFAG}.$

\end{problem1}
\begin{problem2}\label{neut-pb}
Determine the groups in  ${\rm SPFNG}_h(n)$, ${\rm SPFNG}_h$,  ${\rm SPFNG}(n)$ and  ${\rm SPFNG}.$
\end{problem2}
\begin{problem3}\label{sym-pb}
Determine the groups in ${\rm SPFSG}_h(n)$, ${\rm SPFSG}_h$,  ${\rm SPFSG}(n)$ and ${\rm SPFSG}.$
\end{problem3}

Note that the anonymity  problem  and the neutrality  problem mainly correspond, in the context of  {\sc spf}s, to two questions posed by Kelly  in \cite{Kelly91}. The symmetry problem  is instead completely new.
We will see that those three problems are in no way immediately  deducible one  by another one and that a serious attack involving concepts of group theory is needed to approach them, with a  different  strategy for each of them.
For the moment, observe that Proposition \ref{legami}\,(v) immediately implies the two inclusions $${\rm SPFSG}_h(n)\cap \{V\times\{id\}: V\leq S_h\}\subseteq {\rm SPFAG}_h(n), $$$$ {\rm SPFSG}_h(n)\cap \{\{id\}\times W :  W\leq S_n\}\subseteq {\rm SPFNG}_h(n).$$

\section{First steps on symmetry groups} \label{firstSteps}
In this section we start with the study of the symmetry problem, giving some fundamental inclusions for ${\rm SPFSG}_h(n)$. We begin describing a special class of regular subgroups.
\begin{definition}\label{def_massimaleregolare}
\rm{Let $U\in \mathcal{R}_h(n)$. We say that $U$ is  \emph{ regular maximal} if  $V \in \mathcal{R}_h(n)$ and $V \geq U$ imply $V=U$.
 We denote the set of regular maximal subgroups by $\overset{\textbf{.}}{\mathcal{R}}_h(n)$.}
\end{definition}

\begin{proposition}  \label{max-reg-conj}
  The class $\overset{\textbf{.}}{\mathcal{R}}_h(n)$ is nonempty and conjugacy-closed.
\end{proposition}
\begin{proof} The set $\mathcal{R}_h(n)$ is a nonempty poset, by inclusion, and thus it admits maximal elements which are exactly those forming $\overset{\textbf{.}}{\mathcal{R}}_h(n)$.
Let $U\in \overset{\textbf{.}}{\mathcal{R}}_h(n)$ and $g\in G$. Assume that $U^g\leq V$ for some $V \in \mathcal{R}$. Then, since conjugation is a bijection and a group homomorphism, we have $(U^g)^{g^{-1}}\leq V^{g^{-1}}$ and hence $U\leq V^{g^{-1}}$. By Proposition \ref{R-conj-closed}, we have that $V^{g^{-1}}\in \mathcal{R}_h(n)$ and thus by the maximality of $U$ we have $U=V^{g^{-1}}$. Applying conjugation by $g$ we then obtain $U^g= V.$
\end{proof}

\begin{proposition}  \label{regolarita_rappresentabili} \begin{itemize}
\item[$(i)$] Let $U\in \overset{\textbf{.}}{\mathcal{R}}_h(n)$. Then, for every $F \in \mathcal{F}_h(n)^U$,  we have $G(F)= U$. 
  \item[$(ii)$]  The inclusions $\overset{\textbf{.}}{\mathcal{R}}_h(n)\subseteq {\rm SPFSG}_h(n) \subseteq \mathcal{R}_h(n)$ hold. In particular, ${\rm SPFSG}_h(n)\neq \varnothing.$
  \end{itemize}
\end{proposition}
\begin{proof} Let $F\in \mathcal{F}_h(n)$. Then, by the Definitions \ref{u-sym} and \ref{def_gruppidisimmetria},  we immediately have  that $F \in \mathcal{F}_h(n)^{G(F)}$. Hence $\mathcal{F}_h(n)^{G(F)} \neq \varnothing$ and, by Theorem \ref{teofon}, $G(F) \in \mathcal{R}_h(n)$. This shows that ${\rm SPFSG}_h(n)\subseteq \mathcal{R}_h(n)$. In order to conclude the proof of $(ii)$ it is enough to show $(i)$.
Let then $U\in \overset{\textbf{.}}{\mathcal{R}}_h(n)$. Since $U$ is regular, by Theorem \ref{teofon},  we have $\mathcal{F}_h(n)^U \neq \varnothing$. Pick then  $F \in \mathcal{F}_h(n)^U$. Since ${\rm SPFSG}_h(n)\subseteq \mathcal{R}_h(n)$,  we have that  $G(F)$ is regular and clearly $U\leq G(F)$. Hence, by the maximality of $U$ among the regular subgroups, we deduce $G(F)=U$.
\end{proof}

As a consequence of Proposition \ref{regolarita_rappresentabili}, every symmetry groups must be regular. This represents a very strong restriction and has a great impact on the theory we are developing. 

For instance, consider the voting pair $(h,h)$. Then, by Proposition \ref{vecchia},  the  group $U_{N= H}=\{(\varphi,\varphi) \in S_h \times S_h\}$  is never a symmetry group. Thus, unfortunately, when the set of alternatives coincides with the set of  individuals, the respect of the natural level of symmetry given by the group $U_{N= H}$ is an impossible requirement for a {\sc spf}.

As a first application of Proposition  \ref{regolarita_rappresentabili}, we deduce a necessary condition for the fully symmetric voting pairs.
\begin{proposition}\label{full-sym}
If the voting pair $(h,n)$ is fully symmetric then $\gcd(h,n!)=1$.
\end{proposition}
\begin{proof} Let $(h,n)$ be fully symmetric. Then $G$ is a symmetry group and thus, by Proposition \ref{regolarita_rappresentabili}, $G$ is regular. As a consequence, by Proposition \ref{lemma17}, we have
$\gcd(h,n!)=1$.
\end{proof}

 Note that given $U \in \mathcal{R}$ and taken  $F \in \mathcal{F}^U$  we surely have $G(F) \geq U$. The difficulty is to understand if there are or not peculiar $F \in \mathcal{F}^U$ satisfying  also the 
other inclusion $G(F) \leq U$. The main tool for this part of our  investigation will be given by Theorem \ref{costruzione}. It allows to explicitly construct a $U$-symmetric {\sc spf} just by assigning its values on a system of representatives of the $U$-orbits on $\mathcal{P}$.

We now pass to other considerations about the set ${\rm SPFSG}_h(n)$. Remarkably ${\rm SPFSG}_h(n)$ is conjugacy-closed and, given $F\in \mathcal{F}$ such that $G(F)=U$ and $g\in G$, we  can exhibit an explicit construction of a {\sc spf} having as symmetry group the conjugate $U^g$ of $U$ through $g.$

\begin{definition}\label{Fconj}{ \rm Let $F \in \mathcal{F}$ be such that $G(F)=U$. For $g=(\varphi_g, \psi_g)\in G$, define the {\sc spf} $F_g$  by
\begin{equation*}
F_g(p):=\psi_g F(p^{g^{-1}})
\end{equation*}
for all $p\in \mathcal{P}.$ We call $F_g$ the {\it conjugate} of $F$ by $g.$}
\end{definition}
\begin{proposition} \label{chiusuraConiugioG(F)}  Let $F \in \mathcal{F}$ and $g\in G$. Then $G(F_g)= G(F)^g.$
 In particular ${\rm SPFSG}_h(n)$ is conjugacy closed.
\end{proposition}
\proof $ $
 Let $U:=G(F).$ We show first  that $G(F_g) \geq U^g$. Let $u \in U^g$, so that $u=gu'g^{-1}$ for a certain $u' \in U$. Moreover, $\varphi_u=\varphi_g \varphi_{u'} \varphi_{g^{-1}}$ and $\psi_u=\psi_g \psi_{u'} \psi_{g^{-1}}$. Let $p \in \mathcal{P}$. By the $U$-symmetry of $F$ and the Definition \ref{Fconj}, it follows that
\[
F_g(p^{u})=\psi_g F((p^{u})^{g^{-1}})=\psi_g F(p^{g^{-1}gu'g^{-1}})=\psi_g F(p^{u'g^{-1}})
\]
\[
=\psi_g F((p^{g^{-1}})^{u'})=\psi_g \psi_{u'} F(p^{g^{-1}})=\psi_g \psi_{u'} \psi_g^{-1} F_{g}(p)=\psi_u F_g(p).
\]
Hence $ u \in G(F_g)$.

We now show the other inclusion. Let $v \notin U^g$. Then $g^{-1}vg \notin U$ and since $G(F)=U$, there exists $p \in \mathcal{P}$ such that $F(p^{g^{-1}vg}) \neq \psi_g^{-1}\psi_v\psi_gF(p)$. Consider the preference profile $p^{g} $. We show that $F_g((p^{g})^v) \neq \psi_v F_g(p^{g})$. Actually

\[F_{g}((p^{g})^v)=\psi_g F((p^{vg})^{g^{-1}})=\psi_g F(p^{g^{-1}vg})\neq \psi_g \psi_g^{-1}\psi_v\psi_g F(p)=
\]
\[
\psi_v\psi_g F(p) = \psi_v \psi_g F((p^{g})^{g^{-1}})=\psi_v \psi_g \psi_{g}^{-1} F_g(p^{g})= \psi_v F_g(p^{g})
\]

It follows that  $v \notin G(F_g)$  and hence $G(F_g)=U^g$.
The last part of the statement is now obvious.
\endproof

\begin{corollary}\label{G1conj} Let $F \in \mathcal{F}$ and $g\in G$. Then $G_1(F_g)= G_1(F)^g.$
 In particular ${\rm SPFAG}_h(n)$ is conjugacy closed.
\end{corollary}
\begin{proof} By  Proposition \ref{chiusuraConiugioG(F)}, by \eqref{ovvia} and by observing that $S_h \times \lbrace id \rbrace$ is normal in $G$, we have that 
$$G_1(F_g)=G(F_g)\cap \big(S_h \times \lbrace id \rbrace\big)=G(F)^g\cap \big(S_h \times \lbrace id \rbrace\big)^g=[G(F)\cap \big(S_h \times \lbrace id \rbrace\big)]^g=G_1(F)^g.$$
\end{proof}

\section{Anonymity groups vs symmetry groups} \label{anonimityESymmetry}

In this section, we show that for a subgroup of $S_h\times \{id\}$ to be an anonymity group is equivalent to be a symmetry group. Moreover, we examine the behaviour of the trivial group.
\begin{proposition} \label{simmetria e anonim}
Let $(h,n)$ be a voting pair and $V \leq S_h$. The following facts are equivalent:
\begin{enumerate}
\item[$(a)$] $V \times \{id \}\in {\rm SPFSG}_h(n)$;
\item[$(b)$] $V \times \{id \}\in {\rm SPFAG}_h(n)$.
\end{enumerate}
In other words  we have 
$${\rm SPFSG}_h(n)\cap \{V \times \{id \}:V \leq S_h\}= {\rm SPFAG}_h(n).$$
In particular, $${\rm SPFSG}_h(n)\supseteq{\rm SPFAG}_h(n).$$
\end{proposition}
\begin{proof} Let $U:=V \times \{id\}$. 

\noindent $(a)\Rightarrow (b)$ Let  $U\in  {\rm SPFSG}_h(n)$. Then we have that $U=G(F)$, for some  $F \in \mathcal{F}^U$ and,  using \eqref{ovvia}, we deduce $G(F)=G_1(F)$. Thus $U=G_1(F)\in  {\rm SPFAG}_h(n)$. 
\vspace{2mm}

\noindent $(b)\Rightarrow (a)$ Let $U\in  {\rm SPFAG}_h(n)$. Then there exists $F \in \mathcal{F}^U$  such that $G_1(F) = U$. Starting from $F$, we want to construct  now a new  {\sc spf}  $F'$ such that $G(F')=U$. Fix a system ${\bf p}={(p^j)}_{j=1}^{R}$ of representatives for the $U$-orbits on $\mathcal{P}$.
Since $U\leq S_h \times \{id\}$, we have that the set $\mathcal{K}$ of constant profiles is union of $U$-orbits each formed by a singleton and thus 
 $\mathcal{K}\subseteq \{p^j\mid j\in [R]\}$. By Theorem  \ref{costruzione},  there exists a unique $U$-symmetric {\sc spf}  $F'$ such that 
 \[F'(p^j)=
\begin{cases}
id \quad \quad \text{if } \quad p^j \in  \mathcal{K}\\
F(p^j) \quad \text{otherwise}
\end{cases}
\]
Note that we surely have $G(F') \geq U$.  
Moreover, since each $p\in \mathcal{K}$ must appear among the $p^j$, we have that $F'(p)=id$ for all $p\in \mathcal{K}$.
We claim first that $F'$ coincides with $F$ on  $\mathcal{P}\setminus\mathcal{K}$. Let $p\in \mathcal{P}\setminus\mathcal{K}$. Then there exists $j\in [R]$  with  $p^j\notin \mathcal{K}$ and $\bar{\varphi}\in V$ such that $p=p^{j(\bar \varphi,id)}$. From the fact that $F,F' \in \mathcal{F}^U$ we then get
 \begin{equation*}\label{uno}
  F'(p)=F'(p^{j(\bar \varphi,id)})=F'(p^j)=F(p^j)=F(p^{j(\bar \varphi,id)})=F(p).
  \end{equation*}
  
 We claim next that $G(F')=G_1(F')$. Indeed, let $g=(\varphi,\psi) \in G(F')$ and let $p=p_{id}\in \mathcal{K}.$ Then  $ p^{(\varphi,\psi)}=p_{\psi} \in \mathcal{K}.$ By the definition of $F'$ and by its $U$-symmetry, we deduce that
 $$id=F'(p^{(\varphi,\psi)})=\psi F'(p)=\psi\, id=\psi$$
 and thus $\psi=id$. Hence $g=(\varphi,id)\in G(F')\cap S_h \times \{id\}=G_1(F').$ Since the other inclusion is obvious, we therefore have $G(F')=G_1(F')$.
 
 We claim now that $G_1(F')=U$. The inclusion $G_1(F') \geq U$ follows immediately by $G(F') \geq U$  and by $U\leq S_h\times\{id\}$. Since $S_h\times\{id\}\geq G_1(F')$, in order to show the other inclusion $G_1(F') \leq U$ we show that the complement of $U$ in $S_h\times\{id\}$ is included in the complement of $G_1(F')$ in $S_h\times\{id\}$.
 Let then $(\varphi,id)\notin U$ and show that $(\varphi,id) \notin G_1(F')$. Since $G_1(F)=U$, there exists $p \in \mathcal{P}$ such that $F(p^{(\varphi,id)}) \neq F(p)$. Note that surely $p\notin \mathcal{K}$ because assuming $p\in \mathcal{K}$ we have $p^{(\varphi,id)}=p$ and hence also $F(p^{(\varphi,id)})= F(p)$. 
Then, since $F'$ coincides with $F$ on  $\mathcal{P}\setminus\mathcal{K}$,  we also have that $F'(p^{(\varphi,id)}) \neq F'(p)$, as desired.

It follows that $U=G_1(F')=G(F')$ and thus $U\in {\rm SPFSG}_h(n)$.
\end{proof}
As an immediate remarkable consequence, we have the following fact.
\begin{corollary} \label{co-an-sim}
Every anonymity group is also a symmetry group. 
\end{corollary}

 \begin{example} \label{simmalterno}
 \begin{itemize}
 \item[$(a)$] If $2\leq h \leq n!$, then $A_h \times \{id\}$ is an anonymity group.
\item [$(b)$] For every pair $(h,n)$, the group $S_h \times \{id\}$ is an anonymity group.

\end{itemize}
\end{example}
\proof $(a)$ Since the number of individuals is at most the number of linear orders on the alternatives, there exists
 $p\in \mathcal{P}$ with distinct  components. Then $p^{(\varphi_1,id)}=p^{(\varphi_2,id)}$, for some $\varphi_1, \varphi_2 \in S_h$,  implies $\varphi_1=\varphi_2$.
 
 Let  $U:=A_h \times \{id\}$. We show that  $p^{((12),id)} \notin p^U$. By contradiction, assume that there exists
 $u=(\sigma, id) \in U$,  such that $p^{((12),id)}=p^u$. Then $(12)=\sigma\in A_h,$ a contradiction. Thus $p^1:=p$ e $p^2:=p^{((12),id)}$ are representatives for two distinct orbits of $U$ on $\mathcal{P}$. We complete the list $(p^1,p^2)$, obtaining a system of representatives $(p^i)_{i=1}^R$ for the $U$-orbits. Fix $\sigma\in S_h$ with $\sigma\neq id$. 
 By Theorem \ref{costruzione}, there exists a unique $F \in \mathcal{F}^U$ such that 
 \begin{center}
$F(p^1)=id$, \quad \quad $F(p^2)=\sigma$, \quad \quad $F(p^j)=id$ \quad per \quad $3 \leq j \leq r$. 
\end{center}
Then we have $$F(p^{((12),id)})=F(p^2)= \sigma \neq id=F(p).$$
It follows that $G_1(F) \neq S_h \times \{id\}$. On the other hand, $G_1(F) \geq U$ because $F$ is $U$-symmetric and since $U$ is maximal in $S_h \times \{id\}$, we deduce that $G_1(F)=U$.
\smallskip

$(b)$ Consider the constant \textsc{spf} defined by $F( p)=id$ for all $p \in \mathcal{P}$ and note that $G_1(F)=S_h \times \{id\}$. 
\endproof
 
 We show now that the trivial subgroup $1=\{(id,id)\}$ belongs to ${\rm SPFAG}_h(n), {\rm SPFNG}_h(n)$ and ${\rm SPFSG}_h(n)$ whatever the number of individuals and alternatives are. We emphasize that proving  that the trivial group  is a symmetry group involves arguments of different flavour for the diverse collective choices. Kelly  in \cite{Kelly92} shows that $1$ is a symmetry group for those collective choices associating with profiles, made up by orders, a single alternative.

\begin{proposition}\label{1sym}
For every voting pair  $(h,n)$,  the trivial group is an anonymity group, a neutrality group and a symmetry group.
\end{proposition}

\begin{proof} 
 Let first $n \geq 3$. Given $p \in \mathcal{P}$, define the set $U_p:=\{ i \in H \mid p_i(1)=1\}$ and the {\sc spf} $F$ given by
\[F(p)=\left\{\begin{matrix} id & {\text{\rm if}} & U_p = \varnothing\\ p_{\min U_p}  & {\text{\rm if}} &  U_p \neq \varnothing\end{matrix}\right.
\]
for all $p\in \mathcal{P}.$ Recall that $H=[h].$

We claim that $G_1(F)=1$.  Let $\varphi \in S_h\setminus\{id\}$. We show that there exists  $p \in \mathcal{P}$ such that $F(p^{(\varphi,id)})\neq F(p)$. Since $\varphi \neq id$,  we have \begin{equation*}
H_{\varphi}:=\{ i \in H \mid \varphi^{-1}(i) > i \} \neq \varnothing.
\end{equation*}

 Let $i$ be the minimum of $H_{\varphi}$ and let $j :=\varphi^{-1}(i).$ Note that $j>i$. Consider the profile $p \in \mathcal{P}$ such that $p_i = id$, $p_j=(2 3)$ and $p_k=(1 2)$ for $k \in H \setminus \{i,j\}$. Since $p_k(1)=2$ holds for all $k<i$  while $p_i(1)=1$, it follows that  $F(p)=p_i=id$. Now we show that $F(p^{(\varphi,id)})=p_j$. If $k<i$, then we have $\varphi^{-1}(k)\leq k<i$ and thus  $p^{(\varphi,id)}_k=p_{\varphi^{-1}(k)}=(12).$ On the other hand $p^{(\varphi,id)}_i=p_{\varphi^{-1}(i)}=p_j=(23)$ so that $p^{(\varphi,id)}_i(1)=1$. Thus the minimum of $U_{p^{(\varphi,id)}}$ is given by $i$ and $F(p^{(\varphi,id)})=p_j=(23)\neq id$. 

Let next $n=2$. Call $p\in \mathcal{P}$ ordered if there exists  $j\in [h]$ such that $p_i=id$ for $i\in[j]$ and $p_i=(12)$ for $j+1\leq i\leq h.$
Define the {\sc spf} $F$ by
\[F(p)=\left\{\begin{matrix} id & {\text{\rm if}} & p\  {\text{\rm is ordered}}\\ (12)& {\text{\rm if}} & p\  {\text{\rm is  not ordered}} \end{matrix}\right.
\]

Note that  $F(p)=F(p')$ if and only if  $p, p'\in \mathcal{P}$ are both ordered or both not ordered.
Assume, by contradiction, that there exists $(\varphi,id) \in G_1(F)$ with  $\varphi \neq id$. Consider the set $H_{\varphi}$ as in \eqref{aux}. Let $i$ be the minimum of $H_{\varphi}$ and let $j :=\varphi^{-1}(i)\in H.$ Note that $h\geq j>i$. Define the preference profile $p \in \mathcal{P}$ by $p_k=id$ for $k\in[i]$ and $p_k=(12)$ for $i+1\leq k\leq h.$ Then the number of components of $p$ equal to $id$ is exactly $i$ and $p$ is ordered. We claim that the preference profile $p^{(\varphi,id)}$ is instead not ordered. First note that $p^{(\varphi,id)}_i=p_{\varphi^{-1}(i)}=p_j=(12)$. If $i=1$, that fact immediately implies that $p^{(\varphi,id)}$ is not ordered. Assume next that $i\geq 2$. If $1\leq k\leq i-1$, then we have $\varphi^{-1}(k)\leq k<i$ and thus
$p^{(\varphi,id)}_k=p_{\varphi^{-1}(k)}=id$. Then there is one occurrence of $id$ left out and necessarily appearing as preference relation  $p^{(\varphi,id)}_{\ell}$ of some individual $\ell\in H,$ with $\ell>i.$ Hence $p^{(\varphi,id)}$ is not ordered. 

Thus we have shown that $1=\{id\}\times\{id\}$ is an anonymity group. By Theorem \ref{simmetria e anonim}, 1 is also a symmetry group. Now, by \eqref{ovvia}, $1$ is also a neutrality group.
\end{proof}

\section{The solution of the neutrality problem}\label{sol-neut}

For what concerns neutrality we have a quite surprising result. 
\begin{theorem} \label{neutralita} 
Every  subgroup of $\{id\} \times S_n$ is a neutrality group. In particular, ${\rm SPFNG}_h(n)$ is subgroup-closed and conjugacy-closed.
\end{theorem}
\proof 
Let $U=\{id\} \times W$, for some $W\leq S_n$, be a subgroup of $\{id\} \times S_n$. If  $W=S_n$, then by \eqref{Di} we have that 
 the dictatorship $D_1$ of the individual $1$ satisfies $G_2(D_1)= {id} \times S_n$. Assume next that $W<S_n$, so that 
 $m=|S_n:W|\geq 2$. Let $x_1=id, \dots, x_m \in S_n$ be a set of representatives for the right cosets of $W$ in $S_n$. Fix a profile $p \in \mathcal{P}$ and put $p^i:=p^{(id,x_i)}$ for all $i \in [m]$. We claim that each $p^i$ belongs to a different $U$-orbit. Assume, by contradiction, that there exist two different indices $j,k \in [m]$ and $\psi \in W$ such that $p^{j(id,\psi)}=p^k$. Then, by\eqref{azione}, we have that $p^{(id,\psi x_j)}=p^{(id,x_k)}$ and hence, looking at the first component of the preference profiles in that equality, we get $\psi x_j p_1=x_kp_1$. By cancellation in the group $S_n$, we then get $\psi x_j=x_k$,  a contradiction.
 
Thus, each $p^1, \dots, p^m$ can be chosen as representative for its $U$-orbit. Complete now this list to a ordered set $(p^j)_{j=1}^R\in  \mathfrak{S}(U)$ of rappresentative for all the $U$-orbits. By Corollary \ref{reg-sub-Sn}, we have that $U \in \mathcal{R}$. Hence, by Theorem \ref{costruzione}, there exists a unique $F \in \mathcal{F}^U$ such that $F(p^j)=id$ for all $j \in [m]$.
Surely we have $G_2(F) \geq U$. In order to show that  equality holds  we just need to show that if $\psi \in S_n \setminus W$, then $F(p^{(id, \psi)}) \neq \psi F(p)$. Let  $\psi \in S_n \setminus W$. Then there exist $\sigma \in W$ and $j \neq 1$ such that $\psi=\sigma x_j$ and, from the $U$-symmetry of $F$, we have
$$F(p^{(id,\psi)})=F(p^{(id,\sigma x_j)})=F((p^{(id,x_j)})^{(id,\sigma)})=\sigma F(p^{j})=\sigma\neq \psi=\psi F(p^1)=\psi F(p).$$
\endproof

We note two immediate consequences of Theorem \ref{neutralita} .
\begin{corollary}\label{full-neut}
Every voting pair $(h,n)$ is fully neutral.
\end{corollary}
\begin{corollary}\label{full-neut2}
The following equalities hold:
$${\rm SPFNG}_h(n)=\{\{id\} \times W: W\leq S_n\}= {\rm SPFNG}(n).$$
Moreover,
$${\rm SPFNG}_h={\rm SPFNG}=\{\{id\} \times W: W\  \hbox{is a permutation group}\}.$$ 
\end{corollary}

\section{The anonymity problem}\label{sec-anonim}

The anonymity problem is with no doubt harder than the neutrality problem. One thing is certain: it is false that, for every voting pair $(h,n)$, every subgroup of $S_h \times \{id\}$ is an anonymity group. We propose a crucial example.
Recall that the Klein group is the subgroup  of $S_4$ given by
$$K:=\{id, (12)(34), (13)(24), (14)(23)\}.$$
\begin{proposition} \label{quadrinomio}
Let $(h,n)=(4,2)$ and  $U=K \times \lbrace id \rbrace$ where $K\leq S_4$ is the Klein group. Then $U$ is a regular subgroup which is not a symmetry group nor an anonymity group.
\end{proposition}
\proof 
Denote the two possibile linear orders in two alternatives by the symbols $0,1$ so that $\mathcal{P}= \lbrace 0,1 \rbrace^4$. The group $U$ acts on $\mathcal{P}$ permuting the ordered strings of four elements in $\lbrace 0,1 \rbrace$ and it is easily checked that there are $7$ orbits $\mathcal{O}_i$, for $i\in[7]$. We collect first those concerning preference profiles with  a different number of $0$ and $1$. 
\begin{flushleft}
$\mathcal{O}_1= \lbrace [1111]\rbrace,$
\quad
$\mathcal{O}_2=\lbrace [0000] \rbrace,$
\quad
$\mathcal{O}_3= \lbrace [1110],[0111],[1011],[1101] \rbrace,$
\quad

$ $

$\mathcal{O}_4= \lbrace [1000],[0100],[0010],[0001]\rbrace.$
\end{flushleft}

Then those concerning preference profiles with two $0$ and two $1$.
\begin{flushleft}
$\mathcal{O}_5=\lbrace [1100],[0011] \rbrace, \quad \mathcal{O}_6=\lbrace [1010],[0101] \rbrace, \quad \mathcal{O}_7=\lbrace [0110],[1001] \rbrace.$
\end{flushleft}

$U$ is  regular because $U\leq S_h \times \{ id \}$ and thus $\mathcal{F}^U\neq \varnothing$. Pick $F \in \mathcal{F}^U$. We show that  we necessarily have $G(F)>U$. 
Since $F$ is $U$-symmetric, we have that $F$ assumes the same value on every $U$-orbit. Let $x_i\in \{0,1\}$  be the values assumed by $F$ on $ \mathcal{O}_i$, for $i\in \{5,6,7\}$. Then at least two among the $x_5,x_6,x_7$ must be equal. 
Assume first that $x_5=x_6$ and consider $g:=((23),id)\in G\setminus U$. We show that $g \in G(F)$. Note that if $i \in [7] \setminus \{ 4,5 \}$, then $p \in \mathcal{O}_i$ implies  $p^g \in \mathcal{O}_i$ and thus $F(p^g)=F(p)$.
Moreover, if $p \in \mathcal{O}_5$, then $p^g \in \mathcal{O}_6$ and we have
$
F(p^g)=x_6=x_5=F(p).
$
Similarly if  $p \in \mathcal{O}_6$, then $p^g \in \mathcal{O}_5$ and we have
$
F(p^g)=x_5=x_6=F(p).
$
If $x_5=x_7$ or $x_6=x_7$ the proof is the same using respectively $g:=((13),id)$ or $g:=((12),id)$. 

That shows that $U$ is not a symmetry group. Now, by Corollary \ref{co-an-sim}, $U$ is not an anonymity group too.
\endproof

Kelly in \cite[p.18]{Kelly92} conjectures that every subgroup of $S_h \times \{id\}$ is instead an anonymity group for  a particular type of collective choices, called representatives systems. 

As a consequence of the above example we can make clear a fundamental difference between ${\rm SPFNG}_h(n)$ and ${\rm SPFAG}_h(n)$,  or ${\rm SPFSG}_h(n)$.
\begin{corollary}\label{an-no-sc} 
\begin{itemize}
\item[$(i)$] There exist pairs $(h,n)$ such that the set ${\rm SPFAG}_h(n)$  is not subgroup-closed.
\item[$(ii)$]   There exist pairs $(h,n)$ such that the set ${\rm SPFSG}_h(n)$ is not subgroup-closed.
\end{itemize}
\end{corollary}
\proof Consider $(h,n)=(4,2)$. By Example \ref{simmalterno}\, $(b)$, the group $S_4\times \{id\}$ is an anonymity group while, by Proposition \ref{quadrinomio}, its subgroup $K \times \lbrace id \rbrace$ is not an anonymity group.
In order to deal with  symmetry groups note that, from Proposition \ref{simmetria e anonim}, the group $S_4\times \{id\}$ is a symmetry group while  $K \times \lbrace id \rbrace$ is not.
\endproof

\subsection{The role of small stabilizers and a sufficient condition}

In oder to recognize the groups in ${\rm SPFAG}_h(n)$ it is somewhat important to take into account the stabilizers of preference profiles under the action of $S_h \times \{id\}$. Those stabilizers are immediately described thanks to the following definition. Let $(h,n)$ be a voting pair and let $\{H_1,\dots, H_k\}$ be a partition of $H$ into $2\leq k\leq \min\{h,n!\}$ parts. We call the group $\times_{i\in[k]}S_{H_i}$ a complete intransitive subgroup of $S_h$ in $k$ blocks. 

\begin{remark} \label{stabilizzatori} Let $(h,n)$ be a voting pair. A subgroup $V \times \{id\}$ of $S_h \times \{id\}$ is the stabilizer of some $p\in \mathcal{P}$ if and only if $V$ is a complete intransitive subgroup of $S_h$ in $k$ blocks, for some $2\leq k\leq \min\{h,n!\}$, or $V=S_h.$
\end{remark}

We fix now the following notation.

(\dag)  If $V < S_h$, we set $|S_h:V|=m+1$, $m\geq 1$ and we let $\varphi_0=id, \varphi_1,\dots, \varphi_m\in S_h$ be a set of representatives for the right $m+1$ cosets of $V$ in $S_h$.

\begin{lemma} \label{Vqualsiasi} Let  $V < S_h$, $U= V \times \{id\}$ and $p \in \mathcal{P}$. Then $p^{S_h \times \{id\}}$ splits into the $U$-orbits $\{(p^{(\varphi_j,id)})^U:j\in [m]_0\}$. Those orbits are not necessarily distinct. Up to a reordering of the $\varphi_j$, with $j\in [m]$, there exists a system of representatives for such orbits  given by $\{p^{(\varphi_j,id)}:j \in [s]_0\}$ for a suitable $0\leq s\leq m$.

\end{lemma}
\begin{proof}  The fact that $p^{S_h \times \{id\}}$ splits into $U$-orbits is clear. We need to show, more precisely, that 
\begin{equation}\label{unione-orbite}
p^{S_h \times \{id\}}=\bigcup_{j\in [m]_0} (p^{(\varphi_j,id)})^U.
\end{equation}

Surely, for every $j\in [m]_0$, we have $(p^{(\varphi_j,id)})^U\subseteq (p^{(\varphi_j,id)})^{S_h \times \{id\}}=p^{S_h \times \{id\}}.$ Moreover, if $p^{(\varphi,id)}\in p^{S_h \times \{id\}}$, then there exists $j\in [m]_0$ such that $\varphi=v\varphi_j$ for some $v\in V$. It follows that $p^{(\varphi,id)}=(p^{(\varphi_j,id)})^{(v,id)}\in (p^{(\varphi_j,id)})^U.$ That shows \eqref{unione-orbite}.

Now, for $i\neq j \in [m]_0$, we have $(p^{(\varphi_i,id)})^U=(p^{(\varphi_j,id)})^U$ or $(p^{(\varphi_i,id)})^U\cap(p^{(\varphi_j,id)})^U=\varnothing.$ We can then isolate a subsets of $\{(p^{(\varphi_j,id)})^U:j\in [m]_0\}$ formed by orbits which are distinct and cover all the possible orbits.
Since one representative for the $U$-orbits in $p^{S_h \times \{id\}}$ can surely be chosen equal to $p$, there exists $s$, with $0\leq s\leq m$, such that up to a reorder of the $\varphi_j$, with $j\in [m]$, a system of representatives for the $U$-orbits in $p^{S_h \times \{id\}}$ is given by $\{p^{(\varphi_j,id)}:j \in [s]_0\}$.
\end{proof}

\begin{lemma} \label{stab dentro V} Let  $V < S_h$ and $U=V \times \{id\}$. If $p \in \mathcal{P}$ is such that
$\mathrm{Stab}_{S_h \times \{id\}}(p) \leq U$, then $p^U\neq (p^{(\varphi_j,id)})^U$ for all $j \in [m]$. In particular, the $s+1$ orbits of $U$ inside $p^{S_h \times \{id\}}$ are at least two.
\end{lemma}
\begin{proof}  Let $p\in\mathcal{P}$ be such that $\mathit{Stab_{S_h \times \{id\}}(p)} \leq V \times \{id\}$  and assume, by contradiction, that there exists $j \in [m]$ such that $(p^{(\varphi_j,id)})^U=p^U$. Then we have $p^{(\varphi_j,id)} = p^{(\varphi,id)}$, for some $\varphi\in V$ and therefore, using \eqref{azione}, we get that $(\varphi^{-1} \varphi_j, id) \in \mathrm{Stab}_{S_h \times \{id\}}(p) \leq U$. So $\varphi^{-1}\varphi_j \in V$  and hence also $\varphi_j \in V$, a contradiction.
\end{proof}

 The next theorem is  one of the main result of the paper about the anonymity groups. It expresses a sufficient condition for being a group an anonymity group.
 
\begin{theorem} \label{rappresentabili_Sh}
Let $V < S_h$ be such that there exists $p \in \mathcal{P}$ with $\mathrm{Stab}_{S_h \times \{id\}}(p) \leq V \times \{id\}$. Then $V \times \{id\}$ is an anonymity group.
\end{theorem}

\proof
Let $U:=V \times \{id\}$. By Lemma \ref{Vqualsiasi},  $p^{S_h \times \{id\}}$ splits into the $U$-orbits $\{(p^{(\varphi_j,id)})^U:j\in [m]_0\}$. Moreover, by Lemma \ref{stab dentro V}, we have $p^U \neq p^{(\varphi_j,id)U}$  for all $j \in [m]$ and there are at least two $U$-orbits in $p^{S_h \times \{id\}}$.
However, observe that we cannot exclude $p^{(\varphi_j,id)U}=p^{(\varphi_k,id)U}$, for suitable $j, k\in [m]$ with $j\neq k.$ By Lemma \ref{Vqualsiasi}, up to reordering, there exists $s$, with $ 1\leq s\leq m$, such that $p, p^{(\varphi_1,id)},\dots, p^{(\varphi_s,id)}$ is a list of representatives for the $U$-orbits on $\mathcal{P}$ contained in $p^{S_h \times \{id\}}$.

Complete now to a full system of representatives of the $U$-orbits on $\mathcal{P}$ adding further suitable representatives $p^{s+1},\dots, p^R$ of the remaining $U$-orbits.
Fix next $\sigma\in S_n\setminus\{id\}$ and  assign
\begin{center}
	$F(p)=id$, \quad \quad $F(p^{(\varphi_j,id)})=\sigma$ \quad for $j \in [s]$, \quad \quad $F(p^{t})=id$\quad for  $t \in \{s+1,\dots, R\}.$
\end{center}
By Theorem \ref{costruzione}, there exists a unique $F \in \mathcal{F}^{U}$ extending that assignment and, obviously,
$G_1(F) \geq U$. We show that equality holds. Suppose, by contradiction, that there exists $(\varphi,id) \in G_1(F) \setminus U$.
Then $\varphi=\bar\varphi \varphi_j$, for  some $j \in [m]$ and $\bar\varphi \in V$. We claim that $F(p^{(\varphi_j,id)})=\sigma.$ This follows by the definition of $F$ when $j\in [s]$.
 Assume next that $j\geq s+1$. Then, by Lemma \ref{Vqualsiasi}, there exists $k\in[s]$ such that $(p^{(\varphi_j,id)})^U=(p^{(\varphi_k,id)})^U$. As a consequence, we have $p^{(\varphi_j,id)}=(p^{(\varphi_k,id)})^{(v,id)}$ for some $v\in V.$ Hence, recalling that  $F \in \mathcal{F}^{U}$ and using \eqref{azione}, we have 
$$F(p^{(\varphi_j,id)})=F((p^{(\varphi_k,id)})^{(v,id)})=F(p^{(\varphi_k,id)})=\sigma.$$
 Taking again into account the fact that $F \in \mathcal{F}^{U}$ and \eqref{azione}, we now easily compute
 $$F(p^{(\varphi,id)})=F(p^{(\bar\varphi\varphi_j,id)})=F((p^{(\varphi_j,id)})^{(\bar\varphi,id)})=F(p^{(\varphi_j,id)})=\sigma.$$
On the other hand, since $(\varphi,id) \in G_1(F)$,  we also have $F(p^{(\varphi,id)})=F(p)=id$, a contradiction.
\endproof
Note that the above theorem is not invertible. Indeed let $(h,n)$ be a voting pair with $h>n!$. We know, by Proposition \ref{1sym} , that the trivial group is an anonymity group. On the other hand, every $p\in \mathcal{P}$ admits at least two equal components and thus $\mathrm{Stab}_{S_h \times \{id\}}(p)$ is not trivial.

\begin{corollary} \label{intransitivi anonimi}
 Let $V \leq S_h$  contain a complete intransitive subgroup of $S_h$ in $k$ blocks, with $2\leq k\leq \min\{h,n!\}$. Then $V \times \{id\}$ is an anonymity group. 
\end{corollary}
\proof If $V=S_h$, we know that $S_h \times \{id\}$ is an anonymity group. Assume that $V<S_h$.  By Remark \ref{stabilizzatori}, there exists $p\in \mathcal{P}$ such that $\mathrm{Stab}_{S_h \times \{id\}}(p) \leq V \times \{id\}$ and, by Theorem \ref{rappresentabili_Sh}, we deduce that $V \times \{id\}$ is an anonymity group.
\endproof
\begin{corollary}\label{soli-completi}
Let $H_1, \dots, H_k$ be a partition of $H$ in $k$ parts, with $1 \leq k \leq \min\{h,n!\}$. Then the subgroup  $$(\times_{i=1}^k S_{H_i}) \times \{id\}$$ is an anonymity group.
\end{corollary}

\begin{corollary}\label{anon-hminore}
If $h \leq n!$, then every subgroup of $S_h \times \{id\}$ is an anonymity group.
\end{corollary}
\proof
Since $h \leq n!$, we have $\min\{h,n!\}=h$ and we can consider the complete intransitive group of $S_h$ in $h$ blocks, that is, the trivial group. Now every subgroup of $S_h$ contains the trivial group and thus, by Corollary  \ref{intransitivi anonimi}, it is an anonymity group.
\endproof

Note how the above corollary puts the pathology of the Klein subgroup completely under control. By Proposition \ref{quadrinomio}, we know that $K\times\{id\}\notin {\rm SPFAG}_4(2)$ and now, by Corollary \ref{anon-hminore} we see that $K\times\{id\}\in {\rm SPFAG}_4(3).$ 
But what is surely more important,  we are able to solve a main part of the anonymity problem.
\begin{corollary}\label{anon-n-varia}
${\rm SPFAG}_h=\{U\leq S_h\times\{id\}\}$ and ${\rm SPFAG}=\{V \times \{id\}: V\  \hbox{is a permutation group}\}.$
\end{corollary}
Of course we would be happy to know which are the groups in ${\rm SPFAG}_h(n)$, for every voting pair $(h,n)$ and which are the groups in ${\rm SPFAG}(n)$, for every $n$. But those questions remain open.

As an immediate consequence of the above corollary and of Proposition \ref{simmetria e anonim}, we also get some information about symmetry.
\begin{corollary}\label{sym-n-varia}
${\rm SPFSG}_h\supseteq \{U\leq S_h\times\{id\}\}$ and ${\rm SPFSG}\supseteq\{V \times \{id\}: V\  \hbox{is a permutation group}\}.$
\end{corollary}

\section{A critical example}\label{crt-ex}

The tricky results of the previous section could suggest that the hard part of the symmetry problem could be confined to the anonymity problem. Unfortunately things are more complicated than that. For instance, consider $V\leq S_h$ and $W\leq S_n$. We know, by Theorem \ref{neutralita}, that $\{id\}\times W$ is a neutrality group. Hence, if  $V\times\{id\}$ is an anonymity group, then one could guess that the group $V\times W$ is necessarily a symmetry group. Unfortunately this is not the case. A striking example is given by $V=\{id\}$ and by $W=S_2$ as illustrated in the next proposition.

\begin{proposition} \label{idpers2}
Let $(h,n)=(3,2)$ and $U=\{id\} \times S_2$. Then $U$ is a neutrality group which is not a symmetry group.
\end{proposition}
\proof The fact that $U$ is a neutrality group follows by Theorem \ref{neutralita}. In order to show that $U$ is not a symmetry group, we consider the following subgroups of $G$ properly containing $U$:
 $$X:= \langle (12) \rangle \times S_2, \qquad Y:= \langle (23) \rangle \times S_2, \qquad Z:=\langle (13) \rangle \times S_2.$$ 
 
 We show that 
 \begin{equation}\label{unione}
 \mathcal{F}^U=\mathcal{F}^{X} \cup \mathcal{F}^Y \cup \mathcal{F}^Z.
 \end{equation}
 
 This implies that  $U\notin G(\mathcal{F})$. Indeed, assume that \eqref{unione} holds and pick $F \in \mathcal{F}^U$. Then there exists $W\in\{X,Y,Z\}$ such that $F \in \mathcal{F}^W$ and thus $G(F) \geq W>U.$
 
By  $X,Y,Z \geq U$  we immediately have $\mathcal{F}^{X} \cup \mathcal{F}^Y \cup \mathcal{F}^Z \subseteq \mathcal{F}^U$. In order to show that equality holds true we show that
 \begin{equation}\label{unione2}
 |\mathcal{F}^{X} \cup \mathcal{F}^Y \cup \mathcal{F}^Z|=|\mathcal{F}^U|.
 \end{equation}

By Theorem \ref{costruzione} we have $|\mathcal{F}^U|=2^{R_U}$. Moreover it is easily checked that $R_U=4$, so that $|\mathcal{F}^U|=16$. 
Note now that  $$\langle X,Y \rangle=\langle X,Z \rangle=\langle Y,Z \rangle=G=S_3\times S_2,$$

By Proposition \ref{generato_intersezione}, we have
 $$\mathcal{F}^{X} \cap \mathcal{F}^Y=\mathcal{F}^{Y} \cap \mathcal{F}^Z= \mathcal{F}^{X} \cap \mathcal{F}^Z=  \mathcal{F}^{X} \cap \mathcal{F}^Y \cap \mathcal{F}^Z = \mathcal{F}^G.$$
Hence by inclusion-exclusion principle we obtain
$$|\mathcal{F}^{X} \cup \mathcal{F}^Y \cup \mathcal{F}^Z|=|\mathcal{F}^{X}| + |\mathcal{F}^Y|+| \mathcal{F}^Z| - 2|\mathcal{F}^G|.$$
Since it is easily checked that $R_X=R_Y=R_Z=3$ and $R_G=2$,  we deduce, by Theorem \ref{costruzione}, that 
$$|\mathcal{F}^{X} \cup \mathcal{F}^Y \cup \mathcal{F}^Z|=2^3+2^3+2^3-2\cdot 2^2=16,$$
 which proves \eqref{unione2}.
\endproof

With a little more amount of effort, and with similar arguments, one can show that also for $(h,n)=(4,2)$ the subgroup $\{id\} \times S_2 \leq G=S_4\times S_2$  is not a symmetry group. However it does not seem that the proofs for the cases $(h,n)\in \{(3,2), (4,2)\}$ could be generalized. We propose then an open question.

\begin{openproblem} For which voting pairs $(h,n)$ the subgroup $\{id\} \times S_2 $ is a symmetry group?
\end{openproblem}
Note that, for instance,  $\{id\} \times S_2$ is a symmetry group when $(h,n)=(2,2)$. Indeed, by Proposition \ref{lemma17}, $S_2\times S_2 \notin \mathcal{R}_2(2)$. On the other hand, by Corollary \ref{reg-sub-Sn}, $\{id\} \times S_2\in \mathcal{R}_2(2)$. Hence $\{id\} \times S_2 \in \overset{\textbf{.}}{\mathcal{R}}_2(2)$ and thus, by Proposition \ref{regolarita_rappresentabili}, $\{id\} \times S_2$ is a symmetry group.

We emphasize that a comparison of Proposition \ref{idpers2} with Corollary \ref{co-an-sim}  puts  in evidence the different nature and behaviour of anonymity and neutrality groups. Indeed, in general, we have $ {\rm SPFNG}_h(n)\not\subseteq {\rm SPFSG}_h(n)$ while we always have $ {\rm SPFAG}_h(n)\subseteq {\rm SPFSG}_h(n)$.

\section{Anonymity and representability by Boolean functions} \label{subsec:funzioni booleane}

In this section we observe a main link between the anonymity problem for social preference functions and the so-called representability by Boolean functions. We are confident that this new point of view could shed light on some difficult open problems about Boolean functions.

We recall some essential definitions.	Let $h,k \in \mathbb{N}$, with $h,k\geq 2$. A \textit{$k$-valued Boolean function} is a function $F : \{0,1\}^h \rightarrow \{0, \dots, k-1\}$. A $2$-valued Boolean function is called  also a \textit{Boolean function}. 
We denote by $\mathcal{B}_h(k)$ the set of $k$-valued Boolean function. 
Given $ \varphi \in S_h$ and $ x \in \{0,1\}^h$, we set  \[ x^{\varphi} := (x_{\varphi^{-1}(1)}, \dots, x_{\varphi^{-1}(h)})\in \{0,1\}^h,\]
establishing an action of $S_h$ on $\{0,1\}^h$.

Given $F \in \mathcal{B}_h(k)$, the set 
\[S(F):=\big \{ \varphi \in S_h | \,  F(x^{\varphi})=F(x), \ \forall x \in \{0,1\}^h \big \}\] 
 is a subgroup of $S_h$ called the \textit{invariance group}  of $F$.

A permutation group $V$, is called \textit{$k$-representable} if there exists $h\geq 2$ and $F \in \mathcal{B}_h(k)$ such that 
 $S(F)=V\leq S_h$; representable if it is $k$-representable for some $k\geq2$. The set of $k$-representable subgroups is denoted, in the literature, by ${\rm BGR}(k)$; the set of representable subgroups by ${\rm BGR}.$ 
Note that, if we set ${\rm BGR}_h(k):=\{S(F): F\in \mathcal{B}_h(k)\}$, then we have 
$${\rm BGR}(k)=\bigcup_{h}{\rm BGR}_h(k)$$ and $${\rm BGR}=\bigcup_{h, k}{\rm BGR}_h(k).$$
We add that also the consideration of $${\rm BGR}_h:=\bigcup_{k}{\rm BGR}_h(k)$$ could be of interest.

The representable groups  play an important role in computer science, universal algebra and graph theory and have been investigated since more than thirty years starting with the  paper by Clote and Kranakis \cite{CK}. That paper is seminal for two reasons. It introduces the problem for the first time giving some main tools to attack it and
contains a mistake, discovered by Kisielewicz  \cite{Kisie98}, which originated a famous intriguing question in the theory of representability:  
establishing the nature of the set ${\rm BGR}\setminus {\rm BGR}(2).$  The only known group in ${\rm BGR}\setminus {\rm BGR}(2)$ is the Klein group $K$ but it is unknown even if ${\rm BGR}\setminus {\rm BGR}(2)$ is finite or not. For the moment that set is a pure mystery, which seems to be resilient  also to some massive recent attacks (\cite{Gre10}, \cite{KisieGre14},\cite{KisieGre19}). Remarkably, in \cite{KisieGre19} the authors characterize the finite simple groups in ${\rm BGR}(2).$
We emphasize that those recent results seem to confirm that ${\rm BGR}\setminus {\rm BGR}(2)=\{K\}$.
For instance, Grech in \cite{Gre10} shows that the only regular permutation groups belonging to ${\rm BGR}\setminus {\rm BGR}(2)$ is $K.$
However, a proof for ${\rm BGR}\setminus {\rm BGR}(2)=\{K\}$ seems to be, at the moment, an unattainable goal.

Now we come back to the framework of the social preference functions observing a main link between anonymity groups of {\sc spf}s and representable groups. The heart of the matter is that a group $U \leq S_h \times \{id \} $  acts on  the elements of $\mathcal{P}=(S_n)^h$ just permuting the order of a string of objects selected from a set of $n!$ elements, without taking into account their nature of permutations. 

\begin{proposition}\label{legame-bool}
	Let $h,n,k \in \mathbb{N}$, with $h,n,k\geq 2$ and $k\leq n!$.
Then $$\{V\times\{id\}: V\in {\rm BGR}_h(k)\}\subseteq {\rm SPFAG}_h(n).$$
If $k=n=2$, then equality holds. In particular, $$\{V\times\{id\}: V\in {\rm BGR}_h(2)\}= {\rm SPFAG}_h(2)$$ and $$\{V\times\{id\}: V\in {\rm BGR}(2) \}= {\rm SPFAG}(2).$$ \end{proposition}

\begin{proof} To start with we describe in a convenient way the $k$-valued Boolean function environment, 
choosing the symbols $\{0, \dots, k-1\}$ as names for $k$ distinct permutations in $S_n$ including $0=id$ and $1=(12).$ This makes sense because, by assumption, $k \leq n!$. 
Suppose $V \in {\rm BGR}_h(k)$.
Then there exists $F \in \mathcal{B}_h(k)$, $F : \{0,1\}^h \rightarrow \{0, \dots, k-1\}$, such that $S(F)=V$. We now create $F' \in \mathcal{F}_h(n)$ such that $G_1(F')=V \times \{id\}$. Consider a system of representatives $(p^j)_{j=1}^r$ for the $r \in \mathbb{N}$ orbits of $U:=V \times \{id\}$  on $\mathcal{P}_h(n)$ such that  the first $s\leq r$ components $(p^j)_{j=1}^s$  are representatives for the $U$-orbits on $\{0,1\}^h=\{p\in \mathcal{P}_h(n): p_i\in \{id=0, (12)=1\}, \forall i\in[h]\}$.

By  Theorem \ref{teofon}, there exists $F' \in \mathcal{F}_h(n)^U$ such that $F'(p^j)=F(p^j)$ for  $j \leq s$  and $F'(p^j)=F(p^s)$ for $s<j\leq r$. By definition of $\mathcal{F}_h(n)^U$, we surely have
 $G_1(F') \geq U$. Moreover, it is immediately observed that $F'_{|_{\{0,1\}^h}}=F$ . Now, since $S(F)=V$, for each  $ \varphi \in S_h \setminus V$ there exists $\bar p \in \{0,1\}^h\subseteq \mathcal{P}_h(n)$ such that $F(\bar p ^{(\varphi,id)}) \neq F( \bar p )$. As a consequence, we also have $F'(\bar p ^{(\varphi,id)} ) \neq F'(\bar p)$ and thus $G_1(F') =U$. 
 
 Assume now that $k=n=2$. Then $\{0,1\}=S_2$. Let  $U=V \times \{id\}\in  {\rm SPFAG}_h(2)$, with $U=G_1(F')$ for some $F'\in \mathcal{F}_h(2).$  Then we also have $F' \in \mathcal{B}_h(2)$ and  $S(F')=V$.
 \end{proof}

By Proposition \ref{legame-bool}, we see that the problem of recognizing the anonymity groups for a {\sc spf}  in $2$ alternatives is the same that  the problem of recognizing the $2$-representable groups by a Boolean function. Thus, unfortunately, it is also of the same hardness. However, it is possible that the different strategy inspiring the study of {\sc spf}s could shed light also on $2$-representability.

Actually our paper contributes to a better knowledge of representability in at least two directions. First, methodologically, the action of $G=S_h\times S_2$ on $\mathcal{P}=S_2^h= \{0,1\}^h$ gives a solid mathematical structure to rely on, offering a method to build up Boolean functions as explained in Theorem \ref{costruzione}. 
Second, some theorems formulated for social preference functions can be directly rephrased for Boolean functions. For instance, by Theorem \ref{rappresentabili_Sh} and Proposition \ref{legame-bool}, we deduce the following proposition. 
\begin{proposition}\label{link}
	Let $V \leq S_h$ contain a subgroup of type $S_{H_1}\times S_{H_2}$, for $\{H_1, H_2\}$ a partition of $H=[h].$ 
	Then $V$ is $2$-representable.
\end{proposition}
Note that the groups $S_{H_1}\times S_{H_2}$, for $\{H_1, H_2\}$ a partition of $H=[h]$, form the conjugacy class of maximal intransitive subgroups of $S_h$. When $|H_1|=|H_2|$,  we have $h$ even and $S_{H_1}\times S_{H_2}$ is not maximal in $S_h$ because it is contained in a copy of the maximal subgroup $S_{h/2}\wr S_2$. In other words, Proposition \ref{link} guarantees that the following groups are $2$-representable: $S_m\times S_{h-m}$ for $m\in [h-1]$ and $S_{h/2}\wr S_2$ when $h$ is even. Those facts are also a consequence of \cite[Theorem 3.1, Theorem 5.4]{Kisie98}.

\section{The symmetry problem and the orbit extension} \label{capitolo O(U)} 

In this Section we introduce the concept of orbit extension $O(U)$ of a subgroup  $U$ of  $G$, a special  overgroup of $U$ such that, when $U$ is a symmetry group,  equals $U$ itself. That fact is a main result of the paper and gives a very strong tool for deciding if a subgroup of $G$ is or not a symmetry group. Moreover, due to Proposition \ref{simmetria e anonim}, it gives also a tool for deciding if a subgroup of $G$ included in $S_h\times\{id\}$ is or not an  anonymity group.

We start defining a relation on the set of subgroups of $G$.
\begin{definition} \label{minore}{\rm Let  $U,V \leq G$. We write $U \leq_{\mathcal{P}} V$ if \begin{center}
$p^U \subseteq p^V$, \quad $\forall$ $p \in \mathcal{P}.$\
\end{center}
If  $U \leq_{\mathcal{P}} V$ and $V \leq_{\mathcal{P}} U$ we write $V \cong_{\mathcal{P}} U$ and we say that $U$ and $V$ are {\it orbit equivalent} on $\mathcal{P}.$ }
\end{definition}
Of course $U\leq V$ implies $U \leq_{\mathcal{P}} V$ but the converse does not hold. Observe that $\cong_{\mathcal{P}}$ is indeed an equivalence relation.
The role of the relation $ \leq_{\mathcal{P}}$ is apparent by the following result.

\begin{proposition} \label{orbite_contenute}
Let $U,V \leq G$ with $\langle U,V \rangle \in \mathcal{R}$. Then the following facts hold:
\begin{enumerate}
\item[$(i)$] If $U \leq_{\mathcal{P}} V$, then  $\mathcal{F}^V \subseteq \mathcal{F}^U$.
\item[$(ii)$] If $U \cong_{\mathcal{P}} V$, then $ \mathcal{F}^V = \mathcal{F}^U$.
\end{enumerate}
\end{proposition}

\proof $(i)$ Let $\mathbf{p}_V=(p^j)_{j=1}^{R_V} \in \mathfrak{S}(V)$ be a system of representatives for the $V$-orbits on  $\mathcal{P}$. By $U \leq_{\mathcal{P}} V$, we have $p^{jU} \subseteq p^{jV}$ for all $j \in [R_V]$. Thus for $i,j \in [R_V]$ distinct, we have $p^{jU} \cap p^{iU} \subseteq p^{jV} \cap p^{iV} = \varnothing$. Hence the components of  $\mathbf{p}_V$ are representatives of a subset of $U$-orbits. In particular, we have $R_U \geq R_V$ and the ordered list $\mathbf{p}_V$ can be completed 
giving rise to a system $\mathbf{p}_U=(p^j)_{j=1}^{R_U} \in \mathfrak{S}(U)$ of representatives for the $U$-orbits on  $\mathcal{P}$.
Let $F \in \mathcal{F}^V$ and let $F(p^j)=:q_j \in S_n$ for $j \in [R_U]$. By Theorem \ref{costruzione}, there exists a unique $\tilde{F} \in \mathcal{F}^U$ such that 
\begin{center}
$\tilde{F}(p^{j})=q_j$, \quad $\forall$ $j \in R_U.$
\end{center}

We claim that  $F= \tilde{F}$. Let $p \in \mathcal{P}$. Then there exists a unique $j \in [R_V]$ such that  $p \in p^{jV}$. Let $(\varphi, \psi) \in V$  be such that  $p=p^{j(\varphi,\psi)}$. By the $V$-symmetry of $F$ we get
\begin{equation}
F(p)=F(p^{j(\varphi,\psi)})=\psi q_j. \label{equ}
\end{equation} 
Recalling that $p^{jU}\subseteq p^{jV}$,  two possibilities arise: $p \in p^{jU}$ or $p \in p^{jV}\setminus p^{jU}$. We examine them separately.

Let first $p \in p^{jU}$. Then there exists $(\bar \varphi,\bar \psi) \in U$ such that $p=p^{j(\bar \varphi,\bar \psi)}$. Since $ \tilde{F} \in \mathcal{F}^U$, we have 

\begin{equation}\label{equ2}
\tilde{F}(p)= \tilde{F}(p^{j(\bar \varphi,\bar \psi)})= \bar \psi \tilde{F}(p^j)= \bar \psi q_j.
\end{equation}
On the other hand we have $p^{j(\varphi,\psi)}=p=p^{j(\bar \varphi,\bar \psi)}$ and by the regularity of $\langle U,V \rangle$ and by Lemma \ref{banale-reg},  we obtain $ \psi=\bar \psi$. Hence, by  \eqref{equ} and \eqref{equ2}, we get  $F(p)= \psi q_j=\bar \psi q_j =\tilde F(p)$. 

Let next $p \in p^{jV}\setminus p^{jU}$. Then necessarily $R_U>R_V$ and there exist $k \in \lbrace R_V+1, \dots , R_U \rbrace$ and $(\bar \varphi, \bar \psi) \in U$ such that $p=p^{k(\bar \varphi, \bar \psi)}$. Moreover, since $U \leq_{\mathcal{P}} V$, we have $p^k \in p^U \subseteq p^V=p^{jV}$. Then there exists $(\hat \varphi, \hat \psi) \in V$ such that  $p^k=p^{j(\hat \varphi, \hat \psi)}$. 
By $F \in \mathcal{F}^V$ and $ \tilde{F} \in \mathcal{F}^U$, we then get 
\begin{equation}\label{equ3}
\tilde F(p)= \tilde F(p^{k(\bar \varphi, \bar \psi)})=\bar \psi \tilde F(p^{k})= \bar \psi F(p^k)=\bar \psi F(p^{j(\hat \varphi, \hat \psi)})=\bar \psi \hat \psi q_j.
\end{equation}
On the other hand, we have $p^{j(\varphi,\psi)}=p=p^{k(\bar \varphi, \bar \psi)}=p^{j(\bar \varphi \hat \varphi, \bar \psi \hat \psi)}$ and,  by the regularity of $\langle U,V \rangle$  and by Lemma \ref{banale-reg}, we obtain $\psi=\bar \psi \hat \psi$. Thus, by \eqref{equ} and \eqref{equ3}  we get $F(p)=\psi q_j = \bar \psi \hat \psi q_j= \tilde F(p)$. 

Thus we have proved that  $F=\tilde F$ and hence $\mathcal{F}^V \subseteq \mathcal{F}^U$.

 $(ii)$ By $U \leq_{\mathcal{P}} V$ and $V \leq_{\mathcal{P}} U$, using  $(i)$ we have $\mathcal{F}^V \subseteq \mathcal{F}^U$ and $\mathcal{F}^U \subseteq \mathcal{F}^V$. Hence $\mathcal{F}^V = \mathcal{F}^U$.
\endproof

\begin{definition}\label{def_odivu}
{\rm Let $U \in \mathcal{R}$. We define the set of subgroups of $G$
\[
\mathcal{A}(U):=\big \lbrace V \in \mathcal{R} : V  \geq U , \, \, V \leq_{\mathcal{P}} U \big \rbrace
\]
and the subgroup of $G$
\begin{equation*} 
O(U):=\langle \mathcal{A}(U) \rangle.
\end{equation*} }
\end{definition}

We call $O(U)$ the {\it orbit extension} of $U$.  Note that  $U \in \mathcal{A}(U)$ so that 
\begin{equation} \label{odivu-fact}
U\leq O(U)\leq G.
\end{equation}

It is easily seen that if $g\in G$, then $O(U^g)=O(U)^g$. For the sake of brevity we omit that routine proof. In particular if $U$ is normal in $G$, then also $O(U)$ is normal in $G.$ 

\begin{proposition} \label{regolar O(U)} 
	Let $U\in \mathcal{R}$. Then the following facts hold:
	\begin{enumerate}
		\item[$(i)$] $O(U) \in \mathcal{R}$.
		\item[$(ii)$] $\mathcal{F}^U= \mathcal{F}^{O(U)}$.
		\item[$(iii)$] $\bigcap_{F \in \mathcal{F}^{U}}G(F)=O(U).$
	\end{enumerate} 
\end{proposition}
\begin{proof} $(i)$-$(ii)$
	By Proposition \ref{generato_intersezione}  we have
	\begin{equation}\label{aaa}
	\mathcal{F}^{O(U)}= \mathcal{F}^{\langle \mathcal{A}(U) \rangle} = \bigcap_{V \in \mathcal{A}(U)}\mathcal{F}^V.
	\end{equation}
	Let $V \in \mathcal{A}(U)$. Then we have $V \leq_{\mathcal{P}} U$ and $U \leq V \in \mathcal{R}$. Thus, by Proposition \ref{orbite_contenute} we deduce $\mathcal{F}^U \subseteq \mathcal{F}^V$. Hence, by \eqref{aaa}, by the regularity of $U$ and  by Theorem  \ref{teofon},  we get
	\[
	\mathcal{F}^{O(U)}= \bigcap_{V \in \mathcal{A}(U)}\mathcal{F}^V \supseteq \mathcal{F}^U \neq \varnothing.
	\]
	In particular, $\mathcal{F}^{O(U)} \neq \varnothing$ so that, by Theorem \ref{teofon}, $O(U)$  is regular. Moreover, by $U\leq O(U)$ we also have $\mathcal{F}^{O(U)}\subseteq \mathcal{F}^U$ and therefore $\mathcal{F}^U=\mathcal{F}^{O(U)}$.
	\smallskip
	
	$(iii)$  Define $$S:=\underset{F \in \mathcal{F}^{U}}{\bigcap} G(F)$$ and let  $F \in \mathcal{F}^U$. By  $(ii)$, we have $F \in \mathcal{F}^{O(U)}$ and thus $G(F) \geq O(U)$. It follows that  $S \geq O(U)$. We show that equality holds. Suppose, by contradiction, that 
	$S > O(U)$ and let  $g=(\varphi, \psi) \in S \setminus O(U)$. We observe that $\langle g,U \rangle\in \mathcal{R}$. Indeed by  $U \leq S$ and $g \in S$ it follows that $\langle g,U \rangle \leq S$. Now $S$ is surely regular because included in $G(F)\in \mathcal{R}$. Thus also its subgroup $\langle g,U \rangle$ is regular. If we suppose that $p^g \in p^U$ for all  $p\in \mathcal{P}$, then $\langle g,U \rangle \leq_{\mathcal{P}}U$ and so $\langle g,U \rangle \in \mathcal{A}(U)$. Hence $g \in O(U)$, against the assumption. Thus there exists $\bar p \in \mathcal{P}$ such that  $\bar p^{g} \notin \bar p^U$.
	Since $\bar p$ and $\bar p^{g}$ belong to two different $U$-orbits,  they can be chosen as representative for $\bar p^U$ and $(\bar p^{g})^{U}$. 
	Let then $p^1:=\bar p^{g}$, $p^2:=\bar p$ and complete this list up to  a system $(p^j)_{j=1}^{R_U} \in \mathfrak{S}(U)$ of representatives of the $U$-orbits on $\mathcal{P}$. By Theorem \ref{costruzione}, there exists an unique $F \in \mathcal{F}^U$ such that $F(p^j)=id$ for all $j \in [R_U]$. Since $g \in S$, we have that $g \in G(F)$ and hence
	\begin{equation}\label{prima}
	id=F(p^1)=F(\bar p^{g})= \psi F(\bar p)=\psi F(p^2)=\psi.
	\end{equation}
	Fix now $\sigma\in S_n\setminus\{id\}.$ By Theorem \ref{costruzione}, there exists an unique $\hat F\in \mathcal{F}^U$ such that $\hat F(p^1)= \sigma$ and
	$\hat F(p^{j})=id$ for all $j  \in [R_U]\setminus\{ 1\}$.
	Since $g \in G(\hat{F})$, using \eqref{prima}, we then have
	\[
	\sigma=\hat F(p^1)=\hat F(\bar p^{g})=  \hat F(\bar p)=\hat F(p^2)=id,
	\]
	a contradiction. 
\end{proof}

We are finally ready for the promised necessary condition for symmetry and anonymity.

\begin{corollary} \label{condizione_necess}
If $U$ is a symmetry group, then  $O(U)=U$.
\end{corollary}
\proof $ $
By assumption there exists $F \in \mathcal{F}^U$ such that $G(F)=U$. By Proposition \ref{regolarita_rappresentabili}\,$(ii)$, we know that $U\in \mathcal{R}$.
Then, by Proposition \ref{regolar O(U)}\,$(iii)$, we have  $O(U) \leq U$. By \eqref{odivu-fact}, we also have $U \leq O(U)$ and hence $O(U)=U$.
\endproof
\begin{corollary} \label{necessh}
If $U$ is an anonymity group, then $O(U)=U$.
\end{corollary} 
\proof If 
$U$ is an anonymity group, then $U\leq S_h\times \{id\}$ and by Proposition \ref{simmetria e anonim}, $U$ is also a symmetry group. Thus, by Corollary \ref{condizione_necess}  we get $O(U)=U$.
\endproof

\begin{corollary} \label{O-repres}
	If a permutation group $V$ is $2$-representable, then $O(V\times\{id\})=V\times\{id\}$.
\end{corollary} 
\proof It immediately follows from Proposition \ref{legame-bool} and Corollary \ref{necessh}.
\endproof
At the best of our knowledge the necessary condition for $2$-representability expressed by Corollary \ref{O-repres}, has never been explicitly noticed before in the literature. Some ideas of  formalization similar  to  our $O(U)$ can be found in \cite{HM} and in \cite{HM21},  in the context of $k$-valued Boolean function and treating the problem in terms of a Galois connection. Moreover, some similar formalization  appears in \cite{DS} in the context of relation groups.
Of course, the necessity of taking into account the possibility for different subgroups of $S_h$ to produce the same orbits on $\{0,1\}^h$, and the consequences of that on representability issues, are surely stressed since \cite{CK}. Further main considerations appear in \cite{Kisie98}. In particular, \cite[Theorem 2.2]{Kisie98} guarantees that $V\in{\rm BGR}_h$ if and only if $V$ is maximal among the subgroup of $S_h$ having the same number of orbits on $\{0,1\}^h$.

With no doubt, the Definition \ref{def_odivu} of $O(U)$ can appear kinky and one could ask for a simpler one. Unfortunately, it seems that all the details in it are compulsory  to reach the goals given by Corollaries \ref{condizione_necess} and \ref{necessh}. A little consolation is given by considering, for $U\in \mathcal{R}$, the more manageable subset of $G$
\begin{equation*}
W(U):=\{ g \in G: p^g \in p^U, \quad \forall p \in \mathcal{P} \}.
\end{equation*}
This set has a strong link with $O(U)$ even though, in general, $W(U)\neq O(U).$ However $W(U)$  can greatly help in the calculation of $O(U)$ thanks to the following result.

\begin{proposition} \label{O(V) e U(V)}
Let $U\in \mathcal{R}$. Then the following hold
\begin{enumerate}
\item[$(i)$] $O(U) \leq W(U)\leq G$.
\item[$(ii)$] $W(U) \in \mathcal{R}$ if and only if $W(U)=O(U)$.
\item[$(iii)$] $U \cong_{\mathcal{P}} O(U)$.
\item[$(iv)$] If $U\leq S_h\times \{id\}$, then $W(U)\leq S_h\times \{id\}$ and $W(U)=O(U)$.
\item[$(v)$] $W(U) $ can be not regular and hence different from $O(U)$.
\end{enumerate}
\end{proposition}
\proof $ $
 $(i)$ In order to show that $W(U)\leq G$ it is enough to show that it is product closed. Let $g_1,g_2 \in W(U)$. Fix $p \in \mathcal{P}$. By the definition of $W(U)$  there exist  $u_1,u_2\in U$ such that $p^{g_2}=p^{u_2}$ and $(p^{u_2})^{g_1}=({p^{u_2}})^{u_1}$. It follows that 
$$p^{g_1g_2}= (p^{g_2})^{g_1}= (p^{u_2})^{g_1}=({p^{u_2}})^{u_1}=p^{u_1u_2} \in p^U.$$ 

Let now $Z \in \mathcal{A}(U)$. By definition of $\mathcal{A}(U)$ we have $p^{Z} \subseteq p^{U}$ for all  $p \in \mathcal{P}$. Hence $W(U) \geq Z$ and thus $W(U) \geq \langle \mathcal{A}(U) \rangle= O(U)$.
\smallskip

 $(ii)$ Assume that $W(U)\in \mathcal{R}.$ By the definition of  $W(U)$ we get $p^{W(U)} \subseteq p^{U}$  for all $p \in \mathcal{P}$. Thus $W(U) \leq_\mathcal{P} U$. Moreover, we surely have $U \leq W(U)\in \mathcal{R}$. It follows that  $W(U) \in \mathcal{A}(U)$ and  $W(U) \leq O(U)$. Thus, by  $(i)$, $O(U) = W(U)$.
 
 Conversely, if we have $O(U) = W(U)$, then by proposition \ref{regolar O(U)}\,$(i)$, we have $O(U)\in \mathcal{R}$ and thus $W(U)\in \mathcal{R}.$
\smallskip

$(iii)$ Since $U \leq O(U)$ we surely have  $U \leq_{\mathcal{P}} O(U)$. We show that $O(U) \leq_{\mathcal{P}} U$. Let $g \in O(U)$. Then by $(i)$ we have  $g \in W(U)$ and thus $p^{g} \in p^U$  for all $p \in \mathcal{P}$. Hence $p^{O(U)} \subseteq p^{U}$.
\smallskip

$(iv)$ Let $U\leq S_h\times \{id\}$ and $g=(\varphi, \psi)\in W(U).$ Consider $p\in \mathcal{K}$ a constant profile with $p_i=id$ for all $i\in[h]$. Then $p^U=\{p\}$ and thus $p^g\in p^U$ implies $p^g=p$ so that $p^g_1=p_1$, that is, $\psi id=id.$ Hence $\psi=id$ which says $g\in S_h\times \{id\}$. By $(ii)$ and definition of regularity it follows now $O(U)=W(U)$.

$(v)$ By $(ii)$, we just need to exhibit an example of a group $U$ such that $W(U)\notin \mathcal{R}.$ Consider $(h,n)=(2,2)$ and $U=\{id\} \times S_2$. By Proposition \ref{lemma17}, we know that $G=S_2\times S_2$ is not regular. Hence it is enough to show that $W(U)=G.$
Let $p\in \mathcal{P}$. Then we have
\[p^{((12),id)}= \begin{cases}p \quad \quad \quad \quad \text{if} \quad p_1=p_2 \\
 p^{(id,(12))} \quad 	\text{if} \quad p_1 \neq p_2
 \end{cases}
\]
In any case, we have  $p^{((12),id)} \in p^U$. Thus  $((12),id) \in W(U)$ and therefore $W(U) \geq \langle ((12),id), U \rangle = G$, so that $W(U)=G.$
\endproof

As an application of the concept of orbit extension, we can now complete the discussion about $A_h \times \{id\}$  begun in Example \ref{simmalterno}. The following result gives a complete answer to a problem raised by Kelly in \cite{Kelly91}. We emphasize that the alternating group is widely considered in social choice literature and it is important to decide when $A_h \times \{id\}$ is an anonymity group. However explicit formal proofs are usually missing. For instance, in \cite[p. 91]{Kelly91}, it is studied a particular family of social preference correspondences  for $h=3$ and $n=2$ and it is said that $A_3 \times \{id\}$  is not an anonymity group for them, giving no proof. 
\begin{proposition} \label{conclusione alterno}  Let $(h,n)$ be a voting pair. Then the following facts hold:
\begin{itemize}
\item[$(i)$] If $h \leq n!$, then $O(A_h \times \{id\})=A_h \times \{id\}$.
\item	[$(ii)$] If $h > n!$, then $O(A_h \times \{id\})=S_h \times \{id\}$.
\item[$(iii)$] $A_h \times \{id\}\in  {\rm SPFAG}_h(n)$ if and only if $h \leq n!$.
\end{itemize}
\end{proposition}

\begin{proof} $(i)$ By Example \ref{simmalterno} we know that when $h \leq n!$, then $A_h \times \{id\}$ is an anonymity group. Thus, by Corollary \ref{necessh}, we get $O(A_h \times \{id\})=A_h \times \{id\}.$
\smallskip

$(ii)$ Let $n>h!$ and let $U:=A_h \times \{id\}$ and $V:=S_h \times \{id\}$. We show that $U \cong_{\mathcal{P}} V$. $U \leq_{\mathcal{P}} V$ is trivial. Conversely, fix $p \in \mathcal{P}$ and $(\varphi, id) \in V$. We show that $p^{(\varphi,id)} \in p^U$. If $\varphi \in A_h$, no proof is needed. Assume that $\varphi \notin A_h$. Since $h>n!$ there exist distinct $i,j \in [h]$ such that $p_j=p_i$. Then $p^{((ij),id)}=p$. From $\varphi \notin A_h$ follows $\varphi (ij)  \in A_h$, so \[p^{(\varphi,id)}=(p^{((ij),id)})^{(\varphi,id)}=p^{(\varphi (ij),id)} \in p^U.\]
That means $V \leq_{\mathcal{P}} U$ and $V \cong_{\mathcal{P}} U$. Since $V \in \mathcal{R}$ we have $V \in \mathcal{A}(U)$ and $O(U) \geq V>U$. 
Moreover, by Proposition \ref{O(V) e U(V)}, we have that $O(U)\leq W(U)\leq V$ and hence $O(U)=V.$	
	
\smallskip

$(iii)$ If $h \leq n!$ we invoke Example \ref{simmalterno}. If $n>h!$, then we know by $(ii)$ that $O(U)>U$. Thus, by Corollary 
 \ref{necessh}, $U$ cannot be an anonymity group. To deal with symmetry, simply apply Proposition \ref{simmetria e anonim}.
\end{proof}
Note that, as a consequence of Corollary \ref{O-repres} and Proposition \ref{conclusione alterno} we deduce the well-known fact that $A_h\notin {\rm BGR}_h(2)$, for $h\geq 3$. 
Another interesting consequence is that the converse of Proposition \ref{full-sym} does not hold. Indeed, pick $n=3$ and $h=5.$ Then we have $\gcd(5,3!)=1$, but the voting pair $(h,n)=(5,3)$ is not fully symmetric because, by Proposition \ref{conclusione alterno}, $A_5 \times \{id\}$ is not a symmetry group.
\begin{example} \label{quadrinomio2}
Let $(h,n)=(4,2)$ and  $U=K \times \lbrace id \rbrace$ where $K\leq S_4$ is the Klein group. Then $O(U)=U$.
\end{example}
\proof Throughout the proof we use the same notation introduced in the proof of Proposition \ref{quadrinomio}. Assume, by contradiction, that $O(U)>U$. Since  $U$ is normal in $S_4 \times \lbrace id \rbrace$, we have that  $O(U)$ is normal in  $S_4 \times \lbrace id \rbrace$ too.

 Hence, by the structure of the normal subgroup of $S_4$, we have 
$$O(U)\in \{S_4 \times \lbrace id \rbrace,\ A_4 \times \lbrace id \rbrace\}.$$ 

It follows that  $((123),id) \in O(U)$. Consider now $p=[1100]$ and note that 
\[
p^{((123),id)}=[0110]
\]
while 
\[
p^{U}=\lbrace [1100], [0011] \rbrace.
\]
Thus $p^{((123),id)} \notin p^U$ so that $((123),id)\notin W(U).$ As a consequence, by Proposition \ref{O(V) e U(V)}, we also have $((123),id)\notin O(U),$ a contradiction.
\endproof
Since we know, by Proposition \ref{quadrinomio}, that $K \times \lbrace id \rbrace\notin {\rm SPFSG}_4(2)$, the above example shows that the condition $O(U)=U$ does not guarantee that $U$ is a symmetry group. In other words, Corollary \ref{necessh} is not invertible. 
We are not aware of other examples of $U\leq S_h \times \{id\}$ such that $O(U)=U$, with  $U\notin {\rm SPFSG}_h(2)$. On the other hand, there surely are other pathological cases among subgroups not included in $S_h \times \{id\}$. For instance, consider  the voting pair $(3,2)$ and the group $U=\{id\} \times S_2$. By Proposition \ref{idpers2}, we know that  $U\notin {\rm SPFSG}_3(2)$. However, we have $O(U)=U$.  Indeed, it can be proved that for every voting pair $(h,n)$, every $U \leq \{id\} \times S_n$ realizes $O(U)=U$.  The reason for that relies on  some technical considerations about the splitting of $O(U)$ into the product of subgroups which will be part of future research.
\smallskip

Surely the comprehension of ${\rm SPFSG}_h(n), {\rm SPFSG}_h, {\rm SPFSG}(n), {\rm SPFSG}$ is far to be complete. Our results about them are just the tip of an iceberg which deserve to be discovered.

\end{document}